\documentclass[11pt,leqno]{article}
\usepackage[colorlinks=true]{hyperref}
\usepackage{amsmath,amsthm,amsfonts,amssymb,amscd,hyperref}
\usepackage[capitalize]{cleveref}
\usepackage{graphicx}
\usepackage{soul}
\usepackage{mathtools}

\usepackage{graphics,xcolor,url}
\numberwithin{equation}{section}
\usepackage[utf8]{inputenc}
\usepackage{enumerate}

\usepackage{cases}
\usepackage{hyperref}
\usepackage{epsfig}
\usepackage{mathrsfs}

\DeclareMathOperator\Lin{Lin}

\def \exp{\mathrm{exp}}

\def \cH{{\mathcal H}}

\def \cB{{\mathcal B}}

\def \cF{{\mathcal F}}

\def \diag{\mathrm{diag\, }}

\def\qand{\quad \text{and}\quad}

\def\B{\mathbb B}

\def\C{\mathbb C}
\def\R{\mathbb R}

\def\N{\mathbb N}

\def\Z{\mathbb Z}
\def\cU{\mathcal U}
\def\cF{\mathcal F}

\def\id{\mathrm{id}}

\def\sA{{\mathsf A}}

\def\sa{{\mathsf a}}

\def\cal{\mathcal }

\def\arr{\overleftarrow}

\newcommand\multiindex[1]{{\underline{#1}}}
\newcommand\mk{\multiindex{k}}

\newcommand\norm[1]{\left\lVert#1\right\rVert}

\newtheorem{proposition}{Proposition}[section]

\newtheorem*{theorem*}{Theorem}

\newtheorem*{problem*}{Problem}
\newtheorem{definition}[proposition] {Definition}
\newtheorem{lemma}[proposition] {Lemma}

\newtheorem{theo}{Theorem}

\newtheorem{coro}[theo]{Corollary}

\theoremstyle{remark}
\newtheorem{example}[proposition]{Example}
\newtheorem{remark}[proposition]{Remark}

\newcommand{\Inv}[1]{\overleftarrow{#1}}

\makeatletter

\DeclareTextFontCommand{\emph}{\em\bf}

\newcommand{\Addresses}{{%
		\bigskip

		\footnotesize
		Pierre Berger\par\nopagebreak
		\textit{E-mail address}: \texttt{pierre.berger@imj-prg.fr}

		\medskip

		Bernhard Reinke\par\nopagebreak
		\textit{E-mail address}: \texttt{bernhard.reinke@imj-prg.fr}
		
}}

\begin{document}
\title%
{Parametric linearization of skew products}
\author{Pierre Berger\thanks{IMJ-PRG, CNRS, Sorbonne University, Paris University, partially supported by the ERC project 818737 Emergence of wild differentiable dynamical systems.
},
Bernhard Reinke\thanks{IMJ-PRG, CNRS, Sorbonne University, Paris University, partially supported by the ERC project 818737 Emergence of wild differentiable dynamical systems.
}
}

\date{\today}
\maketitle
\begin{abstract}
  We establish a linearization criterion for skew products of contractions in any dimension.
  We prove their smooth or holomorphic parameter dependence. In the smooth setting, we use the language of tame Fréchet spaces.
We apply our result to the linearization of totally projectively expanding Cantor sets.
\end{abstract}

\tableofcontents

\vspace{2cm}

\section{Introduction}
Linearization is a fundamental problem in dynamics. In the simplest setting of a fixed point $p$ of a differentiable map $f$, one would like
to conjugate the dynamics of $f$ near $p$ to the dynamics of the differential $D_p f$ via a homeomorphism $h$:
\begin{equation*}
  h\circ f= D_pf \circ h\; .
\end{equation*}
By the Hartman-Grotman theorem, see \cite{Hartman_1982}, every contracting fixed point is $C^0$-linearizable. 
There are many results that assert  more regularity for $h$, given extra conditions.  One of the first results in this direction is \cite[pp. XCIX - CV]{Poincare_Oeuvres_I}, in which Poincaré established the analytic case under certain non-resonance conditions, see condition $(\cal R)$ in the next section. The Sternberg linearization theorem for contracting fixed points \cite{sternberg_contractions_1957} shows the smooth case under the same non-resonance conditions.  

These linearization theorems of contracting maps have been generalized in many directions, two of which are the introduction of parameter dependence (see for example \cite{Takens71,sell_smooth_1985}) and the replacement of a contracting map 
by a  skew products of contracting maps (see for example \cite{moreira_yoccoz_2010,guysinsky_katok_normal_1998,berger_normal_2014,Ilyashenko_Romaskevich_2016}). The skew products of contracting maps are interesting since they can be used to provide an atlas of charts for  expanding compact sets in which the dynamics is linear. In bifurcation theory involving non-trivial expanding compact set, it is useful to have a linearization theorem which deals with both 
skew products of contracting maps  and a nice dependence on the dynamics.  This is the subject of our main theorem. While \cite{moreira_yoccoz_2010,berger_normal_2014,Ilyashenko_Romaskevich_2016} deals with one dimensional skew products, our main result will apply to any finite dimensional space. Finite dimensional spaces were studied in  \cite{guysinsky_katok_normal_1998,jonsson_varolin_stable_2002,Kalinin_Sadovskaya_2017} but without the parameter dependence. We will show that the linearization depends smoothly or analytically on the whole space of smooth or analytic skew products. In this aspect our study is finer than \cite{Takens71,sell_smooth_1985} where the dependence was studied only along finite dimensional families. Working with such infinite dimensional parameter spaces should be useful for defining renormalization operators nearby homoclinic tangencies or investigating the prevalence following Sauer-Hunt-York \cite{SHY} of some properties such as the Newhouse phenomenon \cite{Newhouse_1979}.

\section{Main results}
\subsection{Statement of the main theorem}
Let $\B$ be the closed unit ball centered at $0$ in $\R^n$. Let $f \colon \B \rightarrow \B$ be an analytic or smooth contracting diffeomorphism fixing $0$. Assume that $D_0 f = \diag \lambda_i$ is diagonal \emph{without resonances}: 
\begin{itemize}\item[$(\cal R)$] 
for any $1\le i\le n$  and any  multiindex $\mk$ with $|\mk|\neq 1$, we have $|\lambda_{i}|\not=|\lambda^\mk|$.\end{itemize} 
Then Poincar\'e's or   Sternberg's  linearization theorem asserts that $f$  is analytically or respectively smoothly \emph{linearizable}: there exists an analytic or resp. $C^\infty$-diffeomorphism $h$ such that:
\begin{equation*}
  h_f\circ f(x)= D_0f\circ h_f(x)\quad \forall x\in \B\; .
\end{equation*}

We will generalize these theorems in 
 \cref{thm:lin_no_param_smooth} and \cref{cor:lin_holomorphic}. In particular, we make the parameter dependence of these results explicit by noting that $f\mapsto h_f$ depends analytically or respectively tamely smooth  on $f$. %

 In the analytic case, let $\tilde \B\subset \C^n$ be the open unit complex ball centered at $0$ and consider $f \colon \tilde \B \rightarrow \tilde \B$.  Let   $\cal B$ denote the complex Banach space of bounded holomorphic functions from $\tilde \B$ into $\C^n$,   vanishing at 0 and with diagonal differential at $0$.
 Then  Poincaré proved that  $h_f\in \cal B $. We will show that the operator $f\mapsto h_f$ from an open subset of $\cal B$ into $\cal  B$ is holomorphic.   

In the smooth case, we will consider the Fr\' echet space $\cal F$ of $C^\infty$-maps from $\B  $ into $\R^n$, which fix $0$ and whose derivative at $0$ is diagonal. We will show that the operator $f\mapsto h_f$ 
from an open subset  $\cal U \subset \cF$ into $\cal F$  is \emph{smooth}:   there is a   family $(L_{i, f} )_{ i\ge 0}$ of $i$-multilinear maps $L_{i, f} $ of $\cal F$, 
 such that $( f,g_1,\dots, g_i)\in \cal U\times \cal F^i\mapsto  L_{i,  f}(g_1,\dots, g_i)$ is continuous and for every $g\in \cal F$, for the $C^r$-norm it holds:
 \begin{equation*}
    h_{ f+t g} = h_{ f} + \sum_{i=  1}^m L_{i,  f}(g,\dots, g)t^i+o  (t^m ) \quad \text{when } t\to 0\; .
 \end{equation*}
Moreover we will show $f\mapsto h_f$  is \emph{tamely smooth} in the sense of \cite{Hamilton82}  recalled in \cref{sec:frechet}. 
\medskip 

We generalize these results by considering  skew product of contractions.  Let $\sigma: \sA\to \sA$ be a homeomorphism of a compact metric space $\sA$. 
Let $(f_\sa)_{\sa\in \sA}$ be a $C^0$-family of $C^1$-self-maps of $\B$. This defines a skew product: 
\begin{equation*}
  f : (\sa, x)\in  \sA\times \B \mapsto (\sigma(\sa), f_\sa(x))\in  \sA\times \B .
\end{equation*}

We now give a sufficient condition to conjugate $f$ to a linear cocycle. Here are the conditions:
\begin{enumerate}[$(H_1)$]
\item  Each map $f_\sa$ is a contracting diffeomorphism from $\B$
  into $\B$ which fixes $0$,  for every $\sa\in \sA$. 
\item The differential of $f_{\sa}$ at $0$ is diagonal; $Df_\sa(0)=\diag \lambda_{\sa,i}$ for every $\sa\in \sA$.
\item For any $1\le i\le n$  and any 
  multiindex $\mk$ with $|\mk|\neq 1$, we have $|\lambda_{\sa,i}|\not=|\lambda_{ \sa}^\mk|$ for all $\sa \in \sA$,
  and the sign of the difference does not depend on $\sa$.
\end{enumerate}

\begin{theo}
  If $(f_\sa)_{\sa\in \sA}$ is a $C^0$-family of $C^\infty$-maps satisfying
  $(H_1-H_2-H_3)$, then  there exists a unique $C^0$-family $(h_\sa)_{\sa\in \sA}$ of
  $C^\infty$-diffeomorphisms from $\B$ into $\R^n$ such that 
  \begin{equation}
    D_0 h_{\sa} = \id\;, \quad  h_\sa(0) = 0 \quad \text{and} \quad h_{\sigma(\sa)} \circ f_\sa(x) =D_0f_\sa\circ h_\sa(x), \quad x\in \B\; .
    \label{eqn:thm}
  \end{equation}
  Moreover, $h$ depends  tamely smooth   on $f$.
\label{thm:lin_no_param_smooth}
\end{theo}
We will first prove this theorem without parameter dependence in \cref{sec:non_param}, and then in \cref{sec:dependence} the full theorem (after giving a proper definition of
tamely smooth maps).

\begin{remark}
  \label{rmk:h_r_4}
  Since we consider contracting diffeomorphisms, we have $0 < |\lambda_{\sa,i}| < 1$ for every $\sa \in \sA$.
  From that it follows that there is an $r \in \N$ such that: 
  \begin{enumerate}[$(H^r_4)$]
  \item  For any $1\le i\le n$  and any 
    multiindex $\mk$ with $  |\mk|\ge r$, we have $|\lambda_{\sa,i}| > |\lambda_{ \sa}^\mk|$ for all $\sa \in \sA$.
  \end{enumerate}
  Thus $(H_3)$ is an open condition among families satisfying $(H_1-H_2)$.
\end{remark}

As a consequence of the proof, we obtain the following finite regularity result:
\begin{coro}
  For $r \geq 2$, if $(f_\sa)_{\sa\in \sA}$ is a $C^0$-family of $C^r$-maps satisfying
  $(H_1-H_2-H_3-H^r_4)$,
  then  there exists a unique $C^0$-family $(h_\sa)_{\sa\in \sA}$ of
  $C^r$-diffeomorphisms from $\B$ into $\R^n$ such that 
  \begin{equation}
    D_0 h_{\sa} = \id\;, \quad h_\sa(0) = 0 \quad \text{and} \quad h_{\sigma(\sa)} \circ f_\sa(x) =D_0f_\sa\circ h_\sa(x), \quad x\in \B\; .
    \label{eqn:coro_finite_reg}
  \end{equation}
\label{cor:lin_no_param}
Moreover $h$ depends continuously on $f$.%
\end{coro}
Let us now consider the case of holomorphic families. In this setting $\tilde \B$ is the \emph{open} unit ball around 0 in $\C$
and $f_{\sa}$ is a holomorphic map from $\tilde \B$ into $\C^n$. Observe that $(H_1-H_2-H_3)$ also make sense in this setting. We have the
following holomorphic counterpart of \cref{thm:lin_no_param_smooth}:
\begin{theo}\label{coro ana}
  If $(f_\sa)_{\sa\in \sA}$ is a $C^0$-family of holomorphic maps satisfying $(H_1-H_2-H_3)$ and it is a family of contracting biholomorphisms from $\tilde \B$ into $\tilde \B$, then  there exists a unique $C^0$-family $(h_\sa)_{\sa\in \sA}$ of biholomorphic maps from $\tilde \B$ into $\C^n$ such that 
  \begin{equation}
    D_0 h_{\sa} = \id\;, \quad h_\sa(0) = 0 \quad \text{and} \quad h_{\sigma(\sa)} \circ f_\sa(x) =D_0f_\sa\circ h_\sa(x), \quad x\in \tilde \B\; .
    \label{eqn:coro_holomorphic}
  \end{equation}
\label{cor:lin_holomorphic}
Moreover, $(h_{\sa})_{\sa}$ depends holomorphically on $f$.
\end{theo}
We will show this in \cref{sec:holomorphic}.
\begin{remark}
  If each of the $(f_\sa)$ are real analytic (i.e.  $f_{\sa}(\tilde \B \cap \R^n) \subset \R^n$), then also the maps $(h_{\sa})$ are real analytic.
  \label{rem:lin_real_analytic}
\end{remark}
\begin{remark}
  The above results are also valid for non-trivial bundles over $\sA$:
  instead of
  considering a map on the product $\sA \times \B$, we could consider a map on a ball bundle over $\sA$. For the ease of exposition,
  the proof is only done in the case of a trivial bundle.
\end{remark}
\subsection{Applications}
The application which motivated this work regards the linearization of expanding compact sets. %

Let $M$ be a Riemannian manifold of dimension $n$ and let $g$ be a $C^1$-self-map of $M$.
We recall a compact subset $K\subset M$ is   \emph{invariant} by $g$ if $g(K)\subset K$. It is \emph{expanding} if furthermore:
\[\|D_xg(u)\|> 1 \;, \quad \forall  x\in K\text{ and unit vector } u \in T_x M  \; .\]

Note that by compactness of $K$, the latter inequality is uniform. 
Expanding compact subsets are  \emph{structurally stable} (see  \cite{Shub_1969} or \cite[Thm 0.1]{berger_persistence_2010}).  This means that for every   $C^1$-perturbation $\tilde g$ of $g$,   there is a unique embedding $\psi_{\tilde g}  \colon K \hookrightarrow M$ that is $C^0$-close to the canonical inclusion $K\hookrightarrow M$ and   such that
  \[ \psi_{\tilde g} \circ g|K = \tilde g\circ \psi_{\tilde g}\; . \]
The set  $K_{\tilde g}=  \psi_{\tilde g} (K)$ is called the \emph{hyperbolic continuation} of $K$; it is   expanding for $\tilde g $.  

A sub-bundle $E\subset TM|K$ is \emph{$Dg$-invariant} if $Dg(E_x) = E_{g(x)}$ for every $x\in K$. Then $Dg$ induces a map denoted $D[g]$ on the quotient $TM/E$. We say that the expanding compact set $K$ is \emph{projectively hyperbolic} at $E$ if: 
\[ \|D_x[g](v)\| >\| D_xg(u)\| \quad \forall x\in K\;, \quad u\in E_x\;, v\in T_xM/E_x,\; \text{such that } \|u\|=1=\|v\|\; .\]
 Note again that by compactness of $K$, the latter inequality is uniform.  We say that the expanding compact set $K$ is \emph{totally projectively hyperbolic} if there exists a flag of invariant sub-bundles:
 \[\{0\}= E_0 \subsetneq E_1 \subsetneq \cdots \subsetneq E_{n-1}\subsetneq E_n=TM|K\]
 such that for every $1\le i<n$,  the set $K$ is \emph{projectively hyperbolic} at $E_i$. Using cones, one shows that totally  projective hyperbolicity is a $C^1$-open property: for $C^1$-small perturbations $\tilde g$ of $g$,  the hyperbolic continuation of $K$ is also totally projective hyperbolic. Furthermore, this property enables to diagonalize the differential. To this end we consider the \emph{inverse limit $\Inv{K}$ of $K$}: 
 \[\Inv{K} \coloneqq \left\{ \sa=(\sa_i)_{i \in \Z } \in K^\Z \colon g(\sa_{i}) = \sa_{i+1} \text{ for all } i \in \Z \right\}\; ,\] 
on which   $g$ induces a homeomorphism:
\[ \Inv g: \sa=(\sa_i)_{i \in \Z }\in \Inv{K} \mapsto (\sa_{i+1} )_{i \in \Z }\in \Inv{K}\;  . \]
Using cones, the total projective hyperbolicity implies the existence of  a unique splitting $F_{1\sa}\oplus F_{2\sa}\oplus \cdots \oplus F_{n\sa}=T_{\sa_0} M$  formed by invariant line bundles that is adapted to the flag: 
\[D_{\sa_0} g (F_{i\sa}) =F_{i\Inv g(\sa)} \;, \quad E_{i\sa_0} = F_{1\sa} \oplus \dots \oplus F_{i\sa} \quad \forall \sa \in \Inv K\ .\]
One easily shows that this splitting depends continuously on $\sa$. %
Here is a consequence of the main theorem: 
\begin{coro}
  Let $K$ be an expanding Cantor set for a self-map $g\in C^\infty(M,M)$. Assume that:
\begin{enumerate}[(a)]
\item the set  $K$ is totally projectively hyperbolic with invariant splitting 
 $F_{1}\oplus F_{2}\oplus \cdots \oplus F_{n }\to  \Inv K $.
\item    for any $1\le i\le n$  and any   multiindex $\mk= (k_1,\dots, k_n) $ with $|\mk|\neq 1$, it holds
  $\| D_{\sa_0}g| F_i\|  \not= \prod_{\ell =1} ^n \| D_{\sa_0}g| F_{\ell}\|^{k_\ell}$ 
  and the sign of the difference does not depend on $\sa\in  \Inv K $. 
\end{enumerate} 
 Then for $r>0$ sufficiently small,   there is a unique continuous family   $(\varphi_{\sa})_{\sa \in \Inv{K}}$ of $C^\infty$-charts 
 $\varphi_{\sa}$ from the closed $r$-ball  $ B_{\sa_0} (r)$ of   $T_{\sa_0} M $ onto a neighborhood of  $\sa_0\in M$  such that:
 \begin{equation}
    \varphi_{\sa}(0)=\sa_0,\quad D_0 \varphi_{\sa}  =\id\qand 
    \varphi_{\arr g (\sa)} \circ D_{\sa_0} g  =  g \circ  \varphi_{ \sa } \quad \text{on } (D_{\sa_0} g)^{-1} B_{\sa_0} (r)\; .
   \label{eqn:cor_app}
 \end{equation}
  \label{cor:application}
\end{coro}
\begin{remark}
  We will show in \cref{cor:application_refined} of \cref{sec:dependence_application} that the resulting linearization depends \emph{smooth tamely} on $g$.
  In \cref{cor:application_holomorphic} of \cref{sec:holomorphic_application}, we state the holomorphic counterpart of this corollary.
  \label{rem:application_forward}
\end{remark}
\begin{proof}[Proof of \cref{cor:application}]
We are going to  apply \cref{thm:lin_no_param_smooth}  using inverse branches of $g$ to define the fiber dynamics of a  skew product over  $\sA:= \Inv{K}$ endowed with the dynamics $\sigma$ equal to   the inverse of $\Inv g$: 
\[\sigma:  \sa=(\sa_i)_{i \in \Z }\in \sA= \Inv{K} \mapsto (\sa_{i-1} )_{i \in \Z }\in \sA= \Inv{K} \; . \] 

  Since $K$ is a Cantor set, it has a neighborhood  $U$ which is diffeomorphic to an open subset of $\R^n$. Likewise $\Inv K$ is a Cantor set, so the splitting   $F_{1}\oplus F_{2}\oplus \cdots \oplus F_{n }\to  \Inv K $ is trivial. This implies that that there is a continuous family of charts $(\Phi_\sa)_{\sa \in \Inv K}$ from $\B$ into $M$ and such that:
  \begin{itemize}
  \item $\Phi_\sa(0)=\sa_0$,
   \item $D_0\Phi_\sa$ sends the splitting $\R\oplus \cdots \oplus \R=\R^n$ to   $F_{1\sa}\oplus F_{2\sa}\oplus \cdots \oplus F_{n \sa}$ and  $\|D_0\Phi_\sa(e_i)\|$ is constant over $\sa \in \sA$ for every vector $e_i$ in the canonical basis of $\R^n$. 
   \end{itemize}  
   Note that in general, $D_0\Phi_\sa$ will not be an isometry from $\R^n$ with the standard length to $T_{\sa_0} M$. However, with these conditions we have
   that $D_0(\Phi_{ \sa }^{-1} \circ g\circ  \Phi_{\sigma(\sa)})$ is a diagonal matrix whose entries have modulus $\| D_{\sigma(\sa_0)}g| F_i\|$.
   In particular, since $K$ is expanding, all these entries have modulus uniformly greater than 1. We shall consider the inverse of these maps to obtain contractions as stated in $(H_1)$. 
   To this end,  we consider $\psi_\sa: x\in \B\mapsto \Phi_\sa (s \cdot x)$, with   $s>0$ sufficiently small so that $g| \psi_{\sigma(\sa)}(\B)$ is an expanding diffeomorphism onto its image which contains
     $ \psi_{\sa}(\B)$. Then the following map is well-defined:
\[f_\sa:=   (\psi_{ \sa }^{-1} \circ g\circ  \psi_{\sigma(\sa)})^{-1} |\B\; .\]
By our construction, $f_\sa$ fixes 0 and has diagonal differential whose entries have modulus $\|D_{\sigma(\sa_0)}g| F_i\|^{-1}$. For $s>0$ sufficiently small enough, it follows that $f_\sa$ is contracting, so $f_{\sa}$ satisfies $(H_1-H_2)$ of \cref{thm:lin_no_param_smooth}.
Moreover $(H_3)$ follows from hypothesis $(b)$.
Applying \cref{thm:lin_no_param_smooth}, let $(h_{\sa})_{\sa\in \sA}$ be the continuous family of $C^\infty$-diffeomorphism $h_\sa: \B \to \R^n$ such that 
$ h_{\sigma (\sa)}\circ  f_\sa = D_0f_\sa\circ   h_{\sa} $ on $\B$. 
Then $\varphi_{\sa} =  \psi_{a} \circ h^{-1}_a \circ D_0(\psi_a)^{-1}$ satisfies \eqref{eqn:cor_app} on some neighborhood of the zero section of $TM|\Inv{K}$.
In particular, for $r > 0$ sufficiently small, we have that $\varphi_{\sa}$ is defined on $B_{\sa_{0}}(r)$ for all $\sa \in \Inv K$.
The uniqueness of $\varphi$ follows from the uniqueness of $h$.
 \end{proof}

\begin{remark}
  If $f$ is real analytic we can also arrange the charts $\phi_{\sa}$ to be real analytic, this follows from the holomorphic setting below.
  \label{rem:application_real_analytic}
\end{remark}

\section{Proof of main theorem without parameter dependence}
\label{sec:non_param}
The proof of \cref{thm:lin_no_param_smooth} without parameter dependence follows from two lemmas:
 \begin{lemma}
   For $r\geq 2$ and every continuous family of $C^\infty$-maps $(f_{\sa})_{\sa \in \sA}$ satisfying $(H_1-H_2-H_3)$, there exists a unique $C^0$-family of polynomial maps $(h_\sa)_{\sa\in \sA}$ of $\R^n$ with degree $\leq r$ such that for every $\sa\in \sA$:
   \begin{equation}
     h_{\sa}(0) = 0\;, \quad D_0 h_{\sa}= \id \quad \text{and} \quad h_{\sigma(\sa)} \circ f_\sa(x)= D_0f_\sa \circ  h_\sa(x)+o(x^r)\; .
     \label{eqn:formal_smooth}
   \end{equation}
\label{lem:formal_no_param_smooth}
\end{lemma}
\begin{lemma}
  For $r \geq 2$, if $f$ is a continuous family of $C^\infty$-maps satisfies $(H_1-H_2-H_3-H^r_4)$, and furthermore
  $f$ satisfies \[f_{\sa}(x)=  D_0f_\sa (x)+o(x^{r})\; .\]
Then there exists  $\delta > 0$, such that there is a unique
$C^0$-family $(h_\sa)_{\sa\in \sA}$  of $C^\infty$-diffeomorphisms $h_\sa$ from the $\delta$-ball around 0  into $\R^n$ satisfying for following conditions for every $\sa\in \sA$:
\begin{equation}
  h_\sa(0) = 0\;, h(x) = x + o(x^r) \quad \text{and} \quad h_{\sigma(\sa)} \circ f_\sa(x)= D_0f_\sa \circ  h_\sa(x)\; .
  \label{eqn:lin_flat_smooth}
\end{equation}
  \label{lem:lin_flat_smooth}
\end{lemma}

For both of these lemmas, we will first prove a version for finite regularity, and then we explain how to promote it to the $C^{\infty}$ setting.
\begin{proof}[Proof of \ref{thm:lin_no_param_smooth} without parameter dependence]
  By \cref{rmk:h_r_4} there is an $r$ such that $f$ satisfies $(H^r_4)$. 
  For the existence of $h$ on a small ball around $0$, we can do a direct composition of  Lemma~\ref{lem:formal_no_param_smooth} and Lemma~\ref{lem:lin_flat_smooth}.
  As we assume that $f$ is a family of contracting diffeomorphisms on $\B$, we can use the relation
\begin{equation*}
  (D_{0}f_\sa)^{-1} \circ h_{\sigma(\sa)} \circ f_\sa(x) = h_{\sa}(x)
\end{equation*}
to iteratively extend $h$ to $\B$. %

The uniqueness of $h$ can be proved similarly: suppose $\tilde h$ is another family satisfying \eqref{eqn:thm}.
By the uniqueness part of Lemma~\ref{lem:formal_no_param_smooth} we get that $h_{\sa}(x) = \tilde h_{\sa}(x) + o(x^r)$, or
$\tilde h_{\sa}(x) \circ h_{\sa}^{-1}(x) = x + o(x^r)$. By the uniqueness part of Lemma~\ref{lem:lin_flat_smooth} applied to the linear map family,
we see that $\tilde h_{\sa} \circ h_{\sa}^{-1}(x) = x$ in a neighbourhood of $0$, or $\tilde h_{\sa}(x) = h_{\sa}(x)$. By extension, we
also get $\tilde h_{\sa}(x) = h_{\sa}(x)$ on $\B$.
\end{proof}
\subsection{Formal linearization}
Actually, $C^\infty$-smoothness is not required in \cref{lem:formal_no_param_smooth}. So we will show the following version:
\begin{lemma}
   For $r\geq 2$ and every continuous family of $C^r$-maps $(f_{\sa})_{\sa \in \sA}$ satisfying $(H_1-H_2-H_3)$, there exists a unique $C^0$-family of polynomial maps $(h_\sa)_{\sa\in \sA}$ of $\R^n$ with degree $\leq r$ such that for every $\sa\in \sA$:
   \begin{equation}
     h_\sa(0)=0\;, \quad D_0 h_\sa=\id \quad \text{and} \quad h_{\sigma(\sa)} \circ f_\sa(x)= D_0f_\sa \circ  h_\sa(x)+o(x^r)\; . 
     \label{eqn:formal_no_param}
   \end{equation}
\label{lem:formal_no_param}
\end{lemma}
\begin{proof}[Proof of Lemma~\ref{lem:formal_no_param}]
We use the language of $r$-jets in the proof.
  Recall that two $C^r$-maps $g, \tilde g \colon \B \rightarrow \R^n$ have the same $r$-jet if $g(x) = \tilde g(x) + o(x^r)$.
  This is an equivalence relation, the resulting quotient space is the space of $r$-jets $J_r$. Every $r$-jet can
  be represented by a unique polynomial of degree $\leq r$. So we could have stated the lemma also in the language of $r$-jets: there
  is a unique $C^0$-family of $r$-jets satisfying \eqref{eqn:formal_no_param}.
  An advantage of jets is that they form a semigroup under composition, and
  that jets with invertible differential are also invertible in the semigroup. 

  We show by induction on $j$ from $1$ to $r$ that we can find a $C^0$-family $(h_\sa)_{\sa\in \sA}$  of $r$-jets $h_\sa$
  with 
\begin{equation}
  h_\sa(0)=0\;, \quad D_0 h_\sa=\id  \quad \text{and} \quad h_{\sigma(\sa)} \circ f_\sa(x)= D_0f_\sa \circ  h_\sa(x)+o(x^{j}) \; . 
  \label{eqn:rec1}
\end{equation}
For $j=1$, there is nothing to show. Assume that $j\ge 2$.  By induction, if
$(h_{\sa})_{\sa \in \sA}$ is a family of $r$-jets satisfying \eqref{eqn:rec1} 
 for $j-1$, then the composition $(h_{\sigma(\sa)} \circ
f_{\sa} \circ h_{\sa}^{-1})_{\sa \in \sA}$ is well-defined in the semigroup of $r$-jets. 
 As we only want to prove the statement on the level of jets, we can
replace $(f_{\sa})_{\sa \in \sA}$ by $(h_{\sigma(\sa)} \circ f_{\sa} \circ
h_{\sa}^{-1})_{\sa \in \sA}$ with $(h_{\sa})_{\sa \in \sA}$ satisfying
\eqref{eqn:rec1} for $j-1$. So
we can assume that $(f_{\sa})$ satisfies
\begin{equation}
\label{for f}f_{\sa}(x)=  D_0f_\sa (x)+o(x^{j-1})\; .
\end{equation} 
Let us recall for every multiindex $\mk = (k_1, \dots, k_n)$ and $x \in \R^n$, $x^\mk = \prod^n_{i=1} x^{k_i}_{i}$ is a real number. Moreover we denote
$|\mk| = k_1 + \dots + k_n$, $\mk! = \prod^n_{i=1} k_i!$ and $\partial^\mk = \partial^{k_1}_{x_1} \dots \partial^{k_n}_{x_n}$.
So we have that $\partial^\mk f_{\sa}(0)$ is in $\R^n $.   Thus, by \cref{for f}, it holds:
\[f_{\sa}(x)=  D_0f_\sa (x)+\sum_{|\mk|=j} \partial^\mk f_\sa(0)\cdot \frac{x^\mk}{\mk!}+ o(x^{j})\; .\]
Note that $ \partial^\mk f_\sa(0)\ \cdot \frac{x^\mk}{\mk!}$ denotes the scalar multiplication of the scalar $\frac{x^\mk}{\mk!}\in \R$ and the vector $ \partial^\mk f_\sa(0)\in \R^n$. 
Note that $\R^n$ is a commutative ring endowed with the canonical product structure \[x\bullet y= (x_i\cdot y_i)_{1\le i\le n}.\]  The ring $\R^n$ endowed with the scalar multiplication, denoted by $\cdot$, it is a $\R$-algebra. In order to stay with `Taylor-like' development, we will proceed with scalar multiplication on the right, as above.  In these notations we have $Df_\sa(x)= \lambda_\sa\bullet x$, where $\lambda_{\sa}=(\lambda_{\sa,i})_{1\leq i \leq n}$ and $Df_{\sa}(x) = \diag \lambda_{\sa,i}$. Thus:
\[f_\sa(x)=  \lambda_a\bullet  x +\sum_{|\mk|=j} \partial^\mk f_\sa(0)\cdot \frac{x^\mk}{\mk!}+ o(x^{j})\; .\]

We look for a continuous family $(h_\sa)_{\sa\in \sA}$ of polynomials in $\R^n[X_1,\dots,X_n]$ of the form:
\[h_\sa (x) = x+ \sum_{|\mk|=j} q_{\sa, \mk} \cdot \frac{x^\mk}{\mk!},\]
where $q_{\sa, \mk}\in \R^n$. Here likewise  $q_{\sa, \mk} \cdot \frac{x^\mk}{\mk!}$ denote the scalar multiplication of the scalar $\frac{x^\mk}{\mk!}\in \R$ and the vector $q_{\sa, \mk}\in \R^n$. %
Thus  we have:
\begin{equation} \label{eq1 linea}
 D_0f_\sa \circ  h_\sa(x)= 
 \lambda_{\sa}\bullet x+\sum_{|\mk|=j} \lambda_{\sa} \bullet q_{\sa, \mk} \cdot \frac{x^\mk}{\mk!}  +o(x^{j}) 
\end{equation} 
\begin{multline} \label{eq2 linea}
h_{\sigma(\sa)} \circ f_\sa(x) = 
f_\sa(x) + \sum_{|\mk|=j} q_{\sigma(\sa), \mk}\cdot \frac{(\lambda_\sa\bullet x)^\mk}{\mk!}  +o(x^{j}) \\
= f_\sa(x) + \sum_{|\mk|=j} q_{\sigma(\sa), \mk}\cdot  \lambda_\sa^\mk \cdot \frac{x^\mk}{\mk!} +o(x^{j}) \\
= 
\lambda_{\sa}\bullet x+
\sum_{|\mk|=j} \partial^\mk f_\sa(0)\cdot \frac{x^\mk}{\mk!}+ 
\sum_{|\mk|=j} q_{\sigma(\sa), \mk}\cdot \lambda_\sa^\mk \cdot \frac{x^\mk}{\mk!} +o(x^{j}) \; .
\end{multline} 
Then subtracting \cref{eq1 linea} from \cref{eq2 linea} we obtain:
\[
h_{\sigma(\sa)} \circ f_\sa(x)- D_0f_\sa \circ  h_\sa(x)= 
\sum_{|\mk|=j} \partial^\mk f_\sa(0)\cdot \frac{x^\mk}{\mk!}+ 
 \sum_{|\mk|=j} q_{\sigma(\sa), \mk}\cdot \lambda_\sa^\mk \cdot \frac{x^\mk}{\mk!} -  \sum_{|\mk|=j} \lambda_{\sa} \bullet q_{\sa, \mk} \cdot \frac{x^\mk}{\mk!} +o(x^{j}) \; .\]
Hence to prove the lemma, it suffices to show the existence of an $C^0$-family  $(h_\sa)_{\sa\in \sA}$ such that:
\begin{equation} \label{to be proved} 
  \sum_{|\mk|=j} \partial^\mk f_\sa(0)\cdot \frac{x^\mk}{\mk!}+ 
 \sum_{|\mk|=j} q_{\sigma(\sa), \mk}\cdot\lambda_\sa^\mk \cdot \frac{x^\mk}{\mk!} =  \sum_{|\mk|=j} \lambda_{\sa} \bullet q_{\sa, \mk} \cdot \frac{x^\mk}{\mk!}\; .
\end{equation} 
Then with $\lambda_{\sa}^{-1}=  (\lambda_{\sa,i}^{-1})_{1\le i\le n}$ it holds:
\[
 \sum_{|\mk|=j} \lambda_{\sa}^{-1} \bullet  \partial^\mk f_\sa(0)\cdot \frac{x^\mk}{\mk!}+ 
 \sum_{|\mk|=j}\lambda_{\sa}^{-1} \bullet   q_{\sigma(\sa), \mk}\cdot\lambda_\sa^\mk \cdot \frac{x^\mk}{\mk!} =  \sum_{|\mk|=j} q_{\sa, \mk} \cdot \frac{x^\mk}{\mk!}\; .
\]
Comparing coefficients in $x^{\mk}$, it is enough to find functions $q_{\sa,\mk} \in C^{0}(\sA, \R^n)$ for every $|\mk| = j$ with
  \begin{equation}
    \lambda_{\sa}^{-1} \bullet  \partial^\mk f_\sa(0) + 
    \lambda_{\sa}^{-1} \bullet  q_{\sigma(\sa), \mk}\cdot\lambda_\sa^\mk = q_{\sa, \mk}\; . \label{eqn:fix_op_formal}
  \end{equation}
Writing $q_{\sa,\mk}$ in components by $q_{\sa,\mk} = (q_{\sa,\mk,1}, \dots, q_{\sa,\mk,n})$, and similarly for $f$, this means we would like to find functions 
$q_{\sa,\mk,i} \in C^{0}(\sA, \R)$ for every $|\mk| = j, 1 \leq i \leq n$ with 
\[
  \lambda_{\sa,i}^{-1} \cdot \partial^\mk f_{\sa,i}(0) + 
  \lambda_{\sa,i}^{-1} \cdot q_{\sigma(\sa),\mk,i}\cdot\lambda_\sa^\mk = q_{\sa,\mk,i}\; .
\]
By $(H_3)$ we know that we have $|\lambda^{-1}_{\sa,i}| \cdot |\lambda_{\sa}^{\mk}|$ is either smaller than 1 for all $\sa \in A$ or larger than 1 for all $\sa \in A$.
So we can consider the operator
\begin{equation*}
  O_{\mk,i}:(q_{\sa, \mk,i} )_{\sa\in \sA} \mapsto
  \begin{cases}
    \lambda_{\sa,i}^{-1} \cdot (q_{\sigma(\sa), \mk,i}\cdot\lambda_\sa^\mk + \partial^\mk f_{\sa,i}(0))
    &\text{ if }  |\lambda^{-1}_{\sa,i}| \cdot |\lambda_{\sa}^{\mk}| < 1 \\
    (\lambda_{\sigma^{-1}(\sa)}^\mk)^{-1} \cdot (\lambda_{\sigma^{-1}(\sa),i} \cdot q_{\sigma^{-1}(\sa), \mk,i}-  \partial^\mk f_{\sigma^{-1}(\sa),i}(0))& \text { otherwise.} 
  \end{cases}
\end{equation*}

And we let $O_{\mk}$ the operator that is componentwise given by the operators $O_{i,\mk}$.
Now $O_\mk$ is an affine operator with a contracting linear part by $(H_3)$ and compactness of $\sA$. So there is a unique fixed point for $O_\mk$, and the fixed point is a solution for \eqref{eqn:fix_op_formal}. 
 Doing this for every  $|\mk|=r$,  we define  the (unique)  family $(h_{\sa} )_{\sa\in \sA}$  satisfying   
\cref{to be proved}.
The uniqueness of $(h_{\sa})_{\sa \in \sA} \mod o(x^r)$ follows from the uniqueness of the fixed point of a contracting map.
\end{proof}
\begin{remark}
  In fact, $(h_{\sa})_{\sa \in \sA}$ depends continuously on $f$:
  since composition of $r$-jets is continuous, it only remains to show that the $q_{\sa,\mk}$ depend continuously on $f$. For this we note that
  the operators $O_{\mk,i}$ depend continuously on $f$, by \cref{lem:smooth_contraction}, their fixed points depend continuously on $f$ and so does $(h_{\sa})_{\sa \in \sA}$.
  \label{rem:lin_formal_continuous}
\end{remark}
\subsection{Linearization of flat contractions}
We prove below the following finite regularity version of Lemma~\ref{lem:lin_flat_smooth}:
\begin{lemma}\label{lemme a rafiner}
  For $r \geq 2$, if $(f_\sa)_{\sa\in \sA}$ is a $C^0$-family of $C^r$-maps satisfying
  $(H_1-H_2-H_3-H^r_4)$ and furthermore
  $f$ satisfies \[f_\sa(x)=  D_0f_\sa (x)+o(x^{r})\; .\]
then there exists  $\delta > 0$, such that there is a unique
$C^0$-family $(h_\sa)_{\sa\in \sA}$  of $C^r$-diffeomorphisms $h_\sa$ from the $\delta$-ball around 0 into $\R^n$ satisfying for following conditions for every $\sa\in \sA$:
\begin{equation}
  h_\sa(0) = 0\;, \quad  h_{\sa}(x) = x + o(x^r), \quad \text{and} \quad h_{\sigma(\sa)} \circ f_\sa(x)= D_0f_\sa \circ  h_\sa(x)\; .
  \label{eqn:lin_flat}
\end{equation}
  \label{lem:lin_flat}
\end{lemma}
\begin{proof} %
  Let us denote the $\delta$-ball around 0 by $\B_{\delta}$.
  Consider the vector space $W_\delta$ of $C^0$-families $(h_\sa)_{\sa\in \sA}$ of $C^r$-maps from $\B_{\delta}$ into $\R^n$. By $(H_1)$,
  we have that $f_{\sa}(\B_{\delta}) \subset \B_{\delta}$ for $\delta \in (0,1)$. Hence the operator
  \begin{align*}
    T_f \colon W_{\delta} &\rightarrow W_\delta \\
    (h_{\sa})_{\sa \in \sA} & \mapsto (D_0f_{\sa})^{-1} \circ h_{\sigma(\sa)} \circ f_{\sa}
  \end{align*}
  is well-defined. It is obviously linear. Let $V_{\delta}$ be the linear subspace of $C^0$-families of $C^r$-maps on $\B_\delta$ that satisfy $h_{\sa}(x) = o(x^r)$ for all $\sa \in \sA$. Then $T_f$ leaves $V_{\delta}$ invariant as we  assumed that $f$ is $r$-flat.

  We consider on $V_{\delta}$ the norm given by 
  \begin{equation}
    \|g\|_{V} = \norm{D^r g}_{0} = \sup_{x \in \B_{\delta}} \norm{D_x^r g} \; .
    \label{eq:norm_adelta}
  \end{equation}
  Here and in the following we use $\| \cdot \|_{0}$ for the uniform $C^0$ norm on $\B_{\delta}$.
  By Taylor--Lagrange expansion around 0, we see that for $k < r$:
\begin{equation}\label{eqn:delta_bound}  \norm{ D_x^k  g } \leq  \delta\cdot  \|g\|_{V}\; , \quad\forall x\in \B_\delta\; .\end{equation}
This implies that  $(V_\delta, \|\cdot \|_{V})$ is a Banach space. Let us denote by $\iota \in W_\delta$ the constant family of the canonical inclusion $\B_{\delta} \subset \R^n$, so $\iota_{\sa}(x) = x$ for all $\sa \in \sA, x \in \B_{\delta}$.
Note that $T_f$ leaves the affine subspace $\iota + V_{\delta}$ invariant, as $T_f(\iota) \in \iota + V_{\delta}$ by $r$-flatness of $f$. We endow $\iota + V_{\delta}$  with the distance induced by the norm of $V_{\delta}$.  
    We will show that, if $\delta$ is small enough, then $T_f$ is a contraction on the affine space $\iota + V_{\delta}$. The resulting fixed point is then a solution
  for our conjugacy problem.

For every $\sa$, let $\mu_\sa:= \min_i |\lambda_{\sa, i}|$ and  $\Lambda_\sa:= \max_i |\lambda_{\sa, i}|$. By $(H^r_4)$, it holds:
\[\mu_\sa > \Lambda_\sa^r\; .\]
Note that $\|D_0f_\sa\|\le \Lambda_\sa$. By compactness of $\sA$, there is $C_1\in (0,1)$ such that:
\[C_1\cdot \mu_\sa > \Lambda_\sa^r\; .\]
Thus for every $C\in (C_1,1)$, there exists $\delta'>0$ such that   
\[C\cdot \mu_\sa > \|D_x f _\sa\| ^r\; , \quad \forall  x\in \B_{\delta'}\; .\]
Thus  we have:
\begin{equation}\label{def C finite reg}  C > \mu_\sa^{-1} \cdot  \|D_x f _\sa\| ^r\; , \quad \forall  x\in \B_{\delta'}\; .\end{equation} 
Now we consider the Faà di Bruno's formula for $D^r$:
  \begin{equation*}
  D^r  (h_{\sigma(\sa)}\circ f_{\sa})  =  D^r_{f_{\sa}}  h_{\sigma(\sa)}  ( Df_{\sa})^{\otimes r}+  
\sum_{k<  r} D^k_{f_{\sa}}  h_{\sigma(\sa)} (P _{\sa, k})  \end{equation*}
     where $P _{\sa, k} $ is a $k$-tensor whose coefficients are polynomial functions of the $r-1$ one first derivatives of $f$. Hence we can bound these polynomials by  $ \mu_{\sa} \cdot  M$ for a certain $M>0$.  This gives:
  \begin{equation*}
    \|  D^r  (h_{\sigma(\sa)}\circ f_{\sa}) \|_{0}   \le  \| D^r_{f_{\sa}}  h_{\sigma(\sa)}  ( Df_{\sa})^{\otimes r}\|_{0}+  
    \mu_{\sa} \cdot   M  \sum_{k<  r} \| D^k h_{\sigma(\sa)}\|_0 \; .
\end{equation*}

  Consequently on $\B_{\delta}$ with $\delta < \delta'$ we have:
  \[ \|D^r T_f(h)_\sa \|_{0} \le  \mu_{\sa}^{-1}  \| Df_{\sa} |\B_{\delta'}  \|_0^r \cdot \|  D^r  (h_{\sigma(\sa)}) \|_{0}  +      M  \sum_{k<  r} \| D^k h_{\sigma(\sa)}\|_0 \; . \]
  So by \eqref{eqn:delta_bound} we have
  \begin{equation*}
    \|D^r T_f(h)_\sa \|_{0} \leq  (\mu_{\sa}^{-1}  \| Df_{\sa} |\B_{\delta'}  \|_0^r  + r \delta M)  \cdot \|  D^r  (h_{\sigma(\sa)}) \|_{0} 
  \end{equation*}
  or 
  \begin{equation}
    \label{eqn:bound_flat_contraction}
    \|T_f(h) \|_{V} \leq  (C + r \delta M)  \cdot \| h \|_{V} \; .
  \end{equation}
  Hence if $\delta$ is small enough such that $C + r \delta M < 1$, we have that $T_f$ is contracting on $V_{\delta}$.

Hence for $\delta$ small enough, we have a unique solution in $W_{\delta}$ for the conclusion of our Lemma. By compactness of $\sA$ and the local inverse function theorem, we can assure that the solution is a family of $C^r$-diffeomorphisms by restricting to a smaller $\delta$. %

\end{proof}
\begin{remark}
  The operator $T_f$ and the constants $C, \delta, M$ depend continuously on $f$. So by \cref{lem:smooth_contraction}, the map $h$ depends continuously on $f$.
  \label{rem:lin_flat_continuous}
\end{remark}
We are now ready to give proof of our main statement in the setting of finite regularity.
\begin{proof}[Proof of Corollary~\ref{cor:lin_no_param}]
  This is a direct composition of Lemma~\ref{lem:formal_no_param} and Lemma~\ref{lem:lin_flat} and the extension argument we used in the proof of the parameterless version of \cref{thm:lin_no_param_smooth}.
  The continuity follows from \cref{rem:lin_formal_continuous} and \cref{rem:lin_flat_continuous}.
\end{proof}
Let us now mention how to obtain a proof of \cref{thm:lin_no_param_smooth} without parameters from Corollary~\ref{cor:lin_no_param}.
\begin{proof}[Proof of \cref{thm:lin_no_param_smooth} without parameters from Corollary~\ref{cor:lin_no_param}]
  By \cref{rmk:h_r_4} there is an $r$ such that $f$ satisfies $(H^r_4)$. Then $f$ satisfies $(H^q_4)$ for every $q \geq r$. So for every $q \geq r$, we
  get by \cref{cor:lin_no_param} a family of $C^q$-diffeomorphisms $(h_{q,\sa})_{\sa \in \sA}$ satisfying \eqref{eqn:coro_finite_reg}. By uniqueness for $(h_{r,\sa})_{\sa \in \sA}$ the families must all agree. In particular, $(h_{r,\sa})_{\sa \in \sA}$ is in fact smooth.
\end{proof}
\section{Parameter dependence and proof of the application}
In this section, we specify the parameter dependence in the tame smooth and analytic setting. We show how to obtain the
parametric version of our application from the parametric version of our main theorem.
\subsection{Background on tame smooth maps}
\label{sec:frechet}
Let us now specify the norm on the spaces that we consider:
\begin{definition}
  For any $r < \infty$,
  we endow the space of maps $C^r(\B, \R^n)$ with the norm $\|h\|_{r} = \max_{0 \le i \le r} \sup_{x \in \B} \|D_x^{i}h\|$,
  and the space $C^0(\sA,C^r(\B,\R^n))$ with the norm $\|(h_{\sa})_{\sa \in \sA}\|_{r} = \max_{\sa \in \sA} \|h_{\sa}\|_{r}$. 
  \label{def:norms_banach}
\end{definition}
The spaces $C^r(\B,\R^{n})$ and $C^{0}(\sA,C^r(\B,\R^{n}))$ endowed with these norms are Banach spaces.
\begin{definition}
  We endow the spaces $C^\infty(\B,\R^n)$ and $C^0(\sA,C^\infty(\B,\R^n))$  with the family of norms $(\|\cdot\|_{r})_{r\geq 0}$.
  \label{def:norms_frechet}
\end{definition}
The spaces $C^\infty(\B,\R^n)$ and $C^{0}(\sA,C^\infty(\B,\R^{n}))$ endowed with the given families of norms are graded Fréchet spaces:
\begin{definition}
  A \emph{graded Fréchet space} $F$ is a topological vector space endowed with an increasing family of norms $(\|\cdot\|_{r})_{r\geq 0}$ generating the topology of $F$, such
  that every sequence that is Cauchy with respect to all norms $(\|\cdot\|_{r})_{r \geq 0}$ has a limit. If $F$ and $G$ are graded Fréchet spaces, then
  their direct product $F \times G$ is also a graded Fréchet space endowed with the family of norms $\|(f,g)\|_{r} = \|f\|_r + \|g\|_r$.
\end{definition}
\begin{example}
  Every Banach space $\cal B$ can be considered as a graded Fréchet space with the constant family of norms $\| \cdot \|_{r} = \| \cdot \|_{\cal B}$.
\end{example}
\begin{definition}
  Let $F, G$ be graded Fréchet spaces, $W \subset F$ be open. Let $\psi \colon W \rightarrow G$ be a continuous map. We say that $\psi$ is $C^{1}$ if
  there is a continuous function $D\psi \colon W \times F \rightarrow G$ such that for every $w \in W, h \in F$, we have that
  \begin{equation*}
    \lim_{t \rightarrow 0} \frac{\psi(w+th)-\psi(w)}{t} = D\psi(w,h) \; .
  \end{equation*}
  We say that $\psi$ is $C^{r+1}$ if $D\psi$ is $C^r$. We say that $\psi$ is \emph{smooth} if $\psi$ is $C^{r}$ for every finite $r$.

  We say that $\psi$ is \emph{tame} if for every $w \in W$, there is a neighborhood $W' \subset W$ and $d, b \in \N$ such that
  for every $r \geq b$ there exists $C_r > 0$ satisfying $\|\psi(\tilde w)\|_r \leq C_r(1 + \|\tilde w\|_{r+d})$ for all $\tilde w \in W'$.

  We say that $\psi$ is \emph{smooth tame} if $\psi$ is smooth and $D^r\psi$ is tame for all $r < \infty$.
  \label{def:smooth}
\end{definition}
\begin{example}
  Let $F$ be a graded Fréchet space, $W \subset F$ be open and $\cal B$ be a Banach space. Then
  every continuous map $\psi \colon W \rightarrow \cal B$ from an open subset $W$ is tame (see \cite[Example II 2.1.4]{Hamilton82}).
  In particular, if $\psi \colon W \rightarrow \cal B$ is smooth, then every derivative $D^r \psi$ is a continuous map into $\cal B$,
  so it is tame as well. So smooth maps into Banach spaces are always smooth tame.
  \label{exa:smooth_banach}
\end{example}
The following interpolation inequality is useful in establishing tameness:
\begin{lemma}
  There exists a family of positive constants $(I_j)_{j \geq 1}$ such that for every $h \in C^{\infty}(\B,\R^{n})$ we have
 \begin{equation}
   \| h \|_k \leq I_j  \| h \|^{\frac{j-k}{j-1}}_j\\| h \|^{\frac{k-1}{j-1}}_j \quad  \forall 1 \leq k \leq j .
   \label{eqn:interpolation}
 \end{equation}
  \label{lem:interpolation}
\end{lemma}
For a proof see \cite[Theorem II.2.2.1]{Hamilton82}.
\begin{lemma}
  The following map is smooth tame:
  \begin{align*}
    \Upsilon \colon C^\infty(\B,\R^{n}) \times C^\infty(\B,\B) &\rightarrow C^\infty(\B,\R^{n}) \\
    (h,f) &\mapsto (h \circ f) \; .
   \end{align*}
  \label{lem:comp_smooth_tame}
\end{lemma}
\begin{proof}
  We sketch the key argument of \cite[Lemma II.2.3.4]{Hamilton82} that the map is tame. %
  In order to show that $\Upsilon$ is tame in $(h_0,f_0) \in C^\infty(\B,\R^{n}) \times C^\infty(\B,\B) $, we can work in a $C^1$-neighborhood $\cal W$ of $(h_0,f_0)$ where $\|h\|_1$ and $\|f\|_1$ are uniformly bounded in $\cal W$. We can bound $\|h \circ f\|_0 \leq \|h\|_0 \leq \|h\|_1$. So in order to show tameness, it suffices to bound $\|D^j (h \circ f)\|_0$ for $j \geq 1$ in terms of $C_j(\|h\|_j + \|f \|_j)$ for some constant $C_j > 0$.

We will again use Faà di Bruno's formula, but now in a more explicit fashion: there are constants
$c_{j,k,i_1,\dots,i_k} \in \Z$ such that
  \begin{equation}
  D^j  (h \circ f)  =   
  \sum_{k \leq  j} \sum_{i_1 +  \dots + i_k = j} c_{j,k,i_1,\dots,i_k} D^k_{f}  h \left( D^{i_1}f \otimes \dots  \otimes D^{i_k}f \right) \; .
  \label{eqn:faa_di_bruno}
\end{equation}
(with $c_{j,j,1,\dots,1} = 1$).
By \cref{lem:interpolation} and the fact that $\|h\|_1$ and $\|f\|_1$ are uniformly bounded for $(h,f) \in \cal W$, there exists constants $I'_j > 0$ (depending on $\cal W$, but not on $(h,f) \in \cal W$) such that
 \begin{equation}
   \| h \|_k \leq I'_j \| h \|^{\frac{k-1}{j-1}}_j\; , \quad
   \| f \|_k \leq I'_j \| f \|^{\frac{k-1}{j-1}}_j \quad  \forall 1 \leq k \leq j, (h,f) \in \cal W\; .
   \label{eqn:interpolation_2}
 \end{equation}
 This allows to bound the summands via  
 \begin{equation}
   \|D^k_{f}  h \left( D^{i_1}f \otimes \dots  \otimes D^{i_k}f \right) \| \leq I'^{k+1}_j \| h \|_j^{\frac{k-1}{j-1}} \|f\|_{j}^{\frac{j-k}{j-1}} 
   \leq I'^{k+1}_j (\|h\|_j + \|f \|_j) \; ,
 \end{equation}
 using for the second step the coarse inequality
 \begin{equation}
   x^ty^{1-t} \leq \max(x,y) \leq x + y \quad  \forall x,y \geq 0, t \in [0,1]\; .
   \label{eqn:coarse_amgm}
 \end{equation}
 So we can bound every term of the right hand side of \eqref{eqn:faa_di_bruno} by some constant times $(\|h\|_j + \|f \|_j)$. This shows that $\Upsilon$ is tame. 

To show that $\Upsilon$ is smooth tame, we observe that it is linear in h and its derivative w.r.t. is a a polynomial of derivatives of  h composed with f and f, and so a composition of tame smooth operations. %
\end{proof}
\subsection{Smooth dependence of application}
\label{sec:dependence_application}
We would like to give a parametric version of \cref{cor:application}.   \begin{coro}
    Let $K, g$ as in \cref{cor:application}.
    Then $\varphi$ depends tamely smooth on $g$. More precisely, there exists a neighborhood $V$ of $K$,
    a trivialization $TM|V \cong V \times \R^n \subset \R^n \times \R^n$ and
    a $C^{\infty}$-neighborhood $\cal U$ of $g$ such that
    $\varphi = \varphi_g$ given by \cref{cor:application} extends to a map in $C^0(\Inv K,C^{\infty}(\B, V)) \subset C^0(\Inv K,C^{\infty}(\B, \R^n))$
    and depends tamely smooth on $g \in U$.
    \label{cor:application_refined}
  \end{coro}
  Note that the statement is slightly different to \cref{cor:application} as we linearize $g$ by a linear map on $\B$. This has the advantage that
  the statement is simpler for the perturbation.

  We begin our discussion by showing smooth dependence of hyperbolic continuation:
\begin{lemma}
  Let $K$ be an expanding compact set for a $C^1$-map $g \colon M \rightarrow M$. Then there exists a $C^{1}$-neighborhood $\tilde \cU$ of $g$ such that the hyperbolic continuation
  $\psi_{\tilde g} \colon K \rightarrow M$ depends tame smoothly on $\tilde g \in \tilde \cU$. 
  \label{lem:dependance_expanding}
\end{lemma}
\begin{proof}
  We can find an $\epsilon$ small and a $C^{1}$-neighborhood $\tilde \cU$ of $g$ enough such that the following holds:
  \begin{itemize}
    \item The Riemannian exponential map $\exp_{x}(v)$ is defined for $x \in K, v \in T_xM, \|v\| < \epsilon$
      and the mapping $\exp_{x}$ from the $\epsilon$-ball in $T_xM$ to the $\epsilon$-ball $B_\epsilon(x)$ is
      $(1 + \epsilon)$ bi-Lipschitz.
    \item For every $\tilde g \in \cal U$, $k \in K$, $\tilde g$ restricts to an $(1 + \epsilon)^3$ expanding 
      diffeomorphism on $B_{\epsilon}(k)$ with image containing $B_{\epsilon}(g(k))$. 
  \end{itemize}
  Under this assumptions, we can consider the following space of functions:
  let $\Gamma^{0}(K, TM|K)$ be the space of sections of the tangent vector bundle $TM$ over $K$.
  This is a Banach space with the uniform $C^0$-norm
  $\|s\| = \sup_{k \in K} \| s(k) \|_{T_kM}$. Let $\cal V$ be the $\epsilon$-ball in this Banach space.

  Given $\tilde g \in \tilde \cU$, we can consider the operator 

  \begin{align*}
    S_{\tilde g} \colon \cal V &\rightarrow \cal V \\
    (\psi(k))_{k \in K} &\mapsto \left( \exp^{-1}_k  \circ (\tilde g|B_{\epsilon}(g(k)))^{-1} \circ \exp_{g(k)} \psi(g(k)) \right) \; .
  \end{align*}
  By our assumptions, this operator is well-defined and $(1 + \epsilon)^{-1}$ contracting on $\cal V$. For $g$ we have that the zero
  section $s_0 \colon k \mapsto 0$ is a fixed point. By restricting $\tilde \cU$ further, we can assume that $\| S_{\tilde g} (s_0) \| < \frac{\epsilon}{1 + \epsilon}$. Using the Banach fixed point theorem, we obtain that $S_{\tilde g}$ has a unique fixed point in $\cal V$. As $S_{\tilde g}$ is smooth, we get by \cref{lem:smooth_contraction} that the fixed point depends smoothly on $\tilde g$. From this we get that the hyperbolic continuation depends smoothly on $\tilde g$. By \cref{exa:smooth_banach}, the dependence is tame smooth.
\end{proof}
\begin{lemma}
  Let $K$ be an totally projectively hyperbolic expanding Cantor set for a $C^1$-map $g \colon M \rightarrow M$ with splitting $F_1 \oplus \dots \oplus F_n$ over $\Inv K$. Then there exists a $C^{1}$-neighborhood $\tilde \cU$ such that the hyperbolic continuation of the $F_i$ depends tame smoothly on $\tilde g \in \tilde \cU$.
  \label{lem:dependence_splitting}
\end{lemma}
\begin{proof}
  We do this in two steps: we show that we can realize 
  $E_i$ and $G_i \coloneqq F_i \oplus \dots \oplus F_n$ as fixed points of contracting maps and thus admit a continuation in a small neighborhood of $g$. 
  Then we continue $F_i = E_i \cap G_i$. Note while $E_i$ is defined over $K$, $G_i$ is like $F_i$ defined over $\Inv K$.

  We recall some geometry related to the Grassmannian. We denote by $Gr_{n,i}$ the set of $i$-dimensional
  linear subspaces of $\R^n$. This is a smooth manifold, if we are given $E \subset \R^n$ $i$ dimensional and a complement $F$ $n-i$ dimensional,
  we have charts by graph transforms
  \begin{align*}
    \xi_{E,F} \colon \Lin(E, F) &\rightarrow Gr_{n,i} \\
    s &\mapsto \left\{ (e, s(e)) \colon e \in E \right\} \; .
  \end{align*}

  For convenience, we fix a smooth trivialization of $TM|V$ for a neighborhood $V$ of $K$. 
  In particular, for $\tilde g$ close to $g$, $v \in V$ with $\tilde g(v) \in V$, we have $D_v \tilde g$
  induces a diffeomorphism $(D_v \tilde g)_{*}$ on $Gr_{n,i}$.

  In particular, for $\tilde g$ close to $g$, we have the following map:
  \begin{align*}
    SE_{\tilde g} \colon C^0(K, Gr_{n,i}) &\rightarrow C^0(K, Gr_{n,i}) \\
    (s(k))_{k \in K} & \mapsto 
    \left( (D_{\psi_{\tilde g}(k)} \tilde g)_{*}^{-1} s(g(k)) \right) \; .
  \end{align*}

  We have that $E_i$ is a fixed point of $SE_{g}$. We want to show that $SE_{\tilde g}$ has a fixed point close to $E_i$.
  For this we will use charts depending on $K$ as follows:

  For $k \in K$, we can consider the orthogonal complement of $E_i$ in $TM_k$
  (with the Riemannian metric of $TM_k$), we denote this by $E^{\perp}_i$. 
  Note that $\Lin(E_i, E^{\perp}_i)$ is a normed vector bundle over $K$. We can consider the space of sections
  $\Gamma^0(K, \Lin(E_i, E^{\perp}_i))$. We have an open embedding $\Xi \colon \Gamma^0(K, \Lin(E_i, E^{\perp}_i)) \hookrightarrow  C^0(K, Gr_{n,i})$,
  sending the zero section to $E_i$. 

  We claim that there is an open neighborhood $W \subset \Gamma^0(K, \Lin(E_i, E^{\perp}_i))$
of the zero section such that after restricting $\tilde \cU$, the composition $(\Xi^{-1} \circ SE_{\tilde g} \circ \Xi)|W$ is well-defined and contracting:
  for $g$, the existence of such a $W$ is equivalent to cone conditions. This open set also works for a further restriction of $\tilde \cU$.
  
  Now $\Xi^{-1} \circ SE_{\tilde g} \circ \Xi$ depends smoothly on $\tilde g$, so we can use \cref{lem:smooth_contraction} to get that the resulting
  fixed point depends smoothly on $\tilde g$. So we have constructed a continuation of $E_i(k)$.

  Very similarly, we can construct a continuation of $G_i$.  We can consider the operator
  \begin{align*}
    SG_{\tilde g} \colon C^0(\Inv K, Gr_{n,n-i + 1}) &\rightarrow C^0(\Inv K, Gr_{n,n-i+1}) \\
    (s(\sa))_{\sa \in \Inv K} & \mapsto 
    \left( (D_{\psi_{\tilde g}(\sa_{-1})} \tilde g)_{*} s(\sigma(\sa)) \right)
  \end{align*}
  such that a fixed point of $SG_{\tilde g}$ close to $G_i$ is the continuation of $G_i$. Now $G_i, G^{\perp}_i, \Lin(G_i, G^{\perp}_i)$
  are vector bundles over $\Inv K$. As for $E_i$, we can see in the ``chart'' of sections $\Gamma^{0}(\Inv K,  \Lin(G_i, G^{\perp}_i))$ that
  $G_i$ is the fixed point of a contracting map that can be continued smoothly.

  For $g$, we know that $G_i(\sa)$ has transverse intersection with $E_i(\sa_0)$ for all $\sa \in \Inv K$. Restricting $\tilde \cU$ further, we can assume
  that the continuation of $G_i(\sa)$ still has transverse intersection with the continuation of $E_i(\sa_0)$, so we obtain the continuation of $F_i(\sa) = G_i(\sa) \cap E_i(\sa_0)$.
  Since the continuations of $G_i$ and $E_i$ depend smoothly on $\tilde g$, so does the continuation of $F_i$. By \cref{exa:smooth_banach}, $F_i$ depends tame smooth on $\tilde g$.
\end{proof}

\begin{proof}[Proof of \cref{cor:application_refined}]
  Using the trivialization $TM|V \cong V \times \R^n \subset \R^n \times \R^n$, let us make the construction in the proof of \cref{cor:application} more explicit: the splitting 
  $F_1 \oplus \dots F_n \rightarrow \Inv K$ is trivial, so there is are continuous unit vector sections $u_i$ of $F_i \rightarrow \Inv K$. In particular, for $\sa \in \Inv K$, there is a unique linear automorphism $L_{\sa}$ of $\R^n$ sending the canonical basis vector $e_i$ to $u_i$.  We can then set 
  \[\Phi_{\sa}(x) = \sa_{0} + L_{\sa}(x) \; .\]
  Continuing with the rest of the proof of \cref{cor:application}, we find a $s > 0$ sufficiently small and such that with
  \[\psi_{\sa}(x) = \Phi_{\sa}(s \cdot x) = \sa_{0} + s L_{\sa}(x)\] we have that
  \[f_\sa:=   (\psi_{ \sa }^{-1} \circ g\circ  \psi_{\sigma(\sa)})^{-1} |\B\]
  is well-defined and satisfies hypotheses $(H_1-H_2-H_3)$ of \cref{thm:lin_no_param_smooth}.

  For a small enough $C^{1}$-neighborhood of $g$, we can continue this construction: By \cref{lem:dependance_expanding} and
  \cref{lem:dependence_splitting}, the line bundle $F_i$ depends tame smoothly on $\tilde g$ in some neighborhood of $g$. So $u_i$, $L_{\sa}$, $\Phi_{\sa}$ and finally $\psi_{\sa}$ depend tame smoothly on $\tilde g$. For a small enough neighborhood of $g$ we have then that the construction of $f_{\sa}$ can be continued and depends tamely smooth on $\tilde g$. 
    By \cref{thm:lin_no_param_smooth}, $h$ depends tamely smooth on $\tilde g$.

    We can now set $\varphi_{\sa} = \psi_{\sa} \circ h^{-1}_{\sa}$  that depends smoothly on $\tilde g$. 

\end{proof}
\subsection{Background on holomorphic maps on Banach spaces}
Let $X, Y$ be complex Banach spaces, let $U \subset X$ be an open subset. 
Let us recall that a continuous map $f \colon U \rightarrow X$ is \emph{holomorphic} if 
for every $u \in U$, the map $f$ is complex Fréchet differentiable in $u$, that is, there is a
complex linear map $D_uf \colon X \rightarrow Y$  with
\begin{equation}
  \lim_{v \rightarrow 0} \frac{\|f(u+v)-f(u) - D_{u}f(v)\|}{|v|} = 0 \; .
\end{equation}
We will state a composition lemma for compositions of holomorphic functions. As above, let $\tilde \B$ be the open
unit ball in $\C^n$ centered at $0$. Let $\cH_{\infty}(\tilde \B, \C^n)$ be the space of bounded holomorphic functions from $\tilde \B$ to $\C^n$.
Endowed with the supremum norm, this is a complex Banach space.
Let $\cH_{cc}(\tilde \B, \tilde \B) \subset \cH_{\infty}(\tilde \B, \C^n)$ be the open ball of radius $1$, so the set of functions from $\tilde \B$ into a compact subball of $\tilde \B$.
\begin{lemma}
  The map
  \begin{align*}
    C \colon \cH_{\infty}(\tilde \B, \C^n) \times \cH_{cc}(\tilde \B, \tilde \B)  &\rightarrow \cH_{\infty}(\tilde \B, \C^n) \\
    (h, f) &\mapsto h \circ f
  \end{align*}
  is holomorphic.
  \label{lem:comp_holom}
\end{lemma}
\begin{proof}
  It is clear that this map is well-defined and continuous. The map is linear in $h$. 
  It is enough to check that the map is holomorphic in $f$. Let $\epsilon > 0$ be such that
  $f(\tilde \B) + \tilde \B_{2\epsilon} \subset \tilde \B$. 
  We claim that for a pair $(f, h) \in \cH_{\infty}(\tilde \B, \C^n) \times \cH_{cc}(\tilde \B, \tilde \B)$ the Fréchet differential in $f$ is given by
  \begin{align*}
    L_{f,h} \colon \cH_{\infty}(\tilde \B, \C^n) &\rightarrow \cH_{\infty}(\tilde \B, \C^n) \\ 
    g(z) &\mapsto D_{f(z)} h (g(z)) \; .
  \end{align*}
  First of all, we can apply the Cauchy integral formula to see that $D h$ is bounded on $f(\tilde \B)$ by $\epsilon^{-2} \| h \|$.
  So $L_{f,h}$ is indeed a bounded operator on $\cH_{\infty}(\tilde \B, \C^n)$.

  If $\|g\| < \epsilon$, we have by the integral Taylor--Lagrange expansion in $h$ that
  \begin{equation}
    \| h \circ (f + g)(z) -  (h \circ f)(z) -  D_{f(z)} h (g(z))\| \leq \frac{1}{2} \|g\|^2\sup_{f(\tilde \B) + \tilde \B_{\epsilon}} \| D^2 h\|  %
  \end{equation}
  But by the Cauchy integral formula, we can bound $\sup_{f(\tilde \B) + \tilde \B_{\epsilon}} \| D^2 h\|$
  by $\epsilon^{-3} \|h\|$. 
  This shows that $L_{f,h}$ is indeed the Fréchet differential. It is complex linear, so the map $C$ is holomorphic.
\end{proof}
\subsection{Holomorphic application}
\label{sec:holomorphic_application}
Let us also mention the holomorphic setting of our application.

For this, let now $M$ be a complex $n$ dimensional manifold and $g$ a holomorphic self-map of $M$. We say that an expanding compact set $K$ of $g$ is \emph{complex totally projectively hyperbolic}
if there exists a flag of invariant complex sub-bundles:
 \[\{0\}= E_0 \subsetneq E_1 \subsetneq \cdots \subsetneq E_{n-1}\subsetneq E_n=TM|K\]
 such that for every $1\le i<n$,  the set $K$ is \emph{projectively hyperbolic} at $E_i$. 

 In this setting, we obtain a splitting
 $F_{1\sa}\oplus F_{2\sa}\oplus \cdots \oplus F_{n\sa}$ of \emph{complex} line bundles  depending continuously on $\sa\in \Inv{K}$. 
\begin{coro} 
  Let $K$ be an expanding Cantor set for a holomorphic self-map $g$ of $M$. Assume that:
\begin{enumerate}[(a)]
\item the set  $K$ is complex totally projectively hyperbolic with splitting 
  $F_{1}\oplus F_{2}\oplus \cdots \oplus F_{n }\to  \Inv K $.
\item    for any $1\le i\le n$  and any   multiindex $\mk=( k_1,\dots, k_n) $ with $|\mk|\neq 1$, we have:
  \[ \| D_{\sa_0}g| F_i\|  \not= \prod_{\ell =1} ^n \| D_{\sa_0}g| F_{\ell}\|^{ k_\ell}\] 
  and the sign of the difference does not depends on $\sa\in  \Inv K $. 
\end{enumerate} 
  
Then for $r > 0$ sufficiently small, there is a unique continuous family  $(\phi_{\sa})_{\sa \in \Inv{K}}$ of biholomorphic charts 
$\phi_{\sa}$ from the closed $r$-ball $B_{a_0}(r)$ of $T_{\sa_0}M$  onto a neighborhood of 
 $a_0 \in M$ such that:

 \begin{equation}
    \varphi_{\sa}(0)=\sa_0,\quad D_0 \varphi_{\sa}  =\id\qand 
    \varphi_{\arr g (\sa)} \circ D_{\sa_0} g  =  g \circ  \varphi_{ \sa } \quad \text{on } B_{\sa_0} (r)\; .
   \label{eqn:cor_app_holomorphic}
 \end{equation}
  \label{cor:application_holomorphic}
  Moreover the same holds true for the hyperbolic continuation of $K$ after holomorphic perturbation of $g$, and the linearization depends homomorphically on the perturbation.  
\end{coro} 
The proof of the corollary is the same as the proof of \cref{cor:application}, but using \cref{cor:lin_holomorphic} instead of the main theorem.

\begin{remark}
  If $M$ has a real structure given by an antiholomorphic involution $\xi_M \colon M \rightarrow M$, and $g$ commutes with $\xi_M$ and $\xi_M$ fixes $K$, then
  $\xi_{M} \circ \varphi_{\sa} \circ \xi_{TM}$ also satisfies \eqref{eqn:cor_app_holomorphic}:
  \begin{equation}
    \xi_M \circ \varphi_{\arr g (\sa)} \circ \xi_{TM} \circ D_{\sa_0} g  = \xi_M \circ \varphi_{\arr g (\sa)} \circ D_{\sa_0} g \circ \xi_{TM} = \xi_M \circ  g \circ  \varphi_{ \sa } \circ \xi_TM = g \circ \xi_{M} \circ \varphi_{\sa} \circ \xi_{TM}.
  \end{equation}
  So by uniquness of $(\varphi_{\sa})_{\sa \in \sA}$, we obtain $\xi_{M} \circ \varphi_{\sa} \circ \xi_{TM} = \varphi_{\sa}$
  From this, we recover real analytic dependence for real analytic maps as promised in \cref{rem:lin_real_analytic}.
\end{remark}
\section{Proof of the parameter dependence of main theorems}
We now give the proofs of \cref{thm:lin_no_param_smooth} and \cref{cor:lin_holomorphic} with full parameter dependence. For \cref{cor:lin_holomorphic}, we work with appropriate Banach spaces of holomorphic maps, while for \cref{thm:lin_no_param_smooth} we have to work with Fréchet spaces. Since the proof of \cref{cor:lin_holomorphic} is closer to the one given in \cref{sec:non_param}, we begin with the holomorphic setting.
\subsection{Holomorphic version of main theorem}
\label{sec:holomorphic}
Here is the holomorphic counterpart of \cref{lem:formal_no_param_smooth}:
\begin{lemma}
   For $r\geq 1$ and every continuous family of holomorphic maps $(f_{\sa})_{\sa \in \sA}$ satisfying $(H_1-H_2-H_3)$, there exists a unique $C^0$-family of polynomial maps $(h_\sa)_{\sa\in \sA}$ of $\C^n$ with degree $\leq r$ such that for every $\sa\in \sA$:
   \begin{equation}
     h_\sa(0)\;, \quad D_0 h_\sa = \id\;, \quad \text{and} \quad h_{\sigma(\sa)} \circ f_\sa(x)= D_0f_\sa \circ  h_\sa(x)+o(x^r) \; .
     \label{eqn:formal_holomorphic}
   \end{equation}
\label{lem:formal_holomorphic}
Moreover $(h_\sa)$ depends holomorphically on $f$.
\end{lemma}
\begin{proof}
  The proof is the holomorphic analogue to the proof of \cref{lem:formal_no_param}, invoking the holomorphic case of
  \cref{lem:smooth_contraction} for holomorphic dependence. %
\end{proof}
Let us now show the holomorphic counterpart of \cref{lem:lin_flat}:
\begin{lemma}
  If $(f_\sa)_{\sa\in \sA}$ is a $C^0$-family of holomorphic maps maps satisfying
  $(H_1-H_2-H_3-H^r_4)$ and furthermore
  $f$ satisfies \[f_\sa(x)=  D_0f_\sa (x)+o(x^{r})\; .\]
then there exists  $\delta > 0$, such that there is a unique
$C^0$-family $(h_\sa)_{\sa\in \sA}$  of biholomorphisms $h_\sa$ on the open $\delta$-ball around $0$ in $\C^n$ satisfying for following conditions for every $\sa\in \sA$:
\begin{equation}
  h_\sa(0)=0\;, \quad h(x)= x + o(x^r)\;, \quad \text{and} \quad h_{\sigma(\sa)} \circ f_\sa(x)= D_0f_\sa \circ  h_\sa(x) \; .
  \label{eqn:lin_flat_holomorphic}
\end{equation}
  \label{lem:lin_flat_holomorphic}
  Moreover, $h$ depends holomorphically on $f$.
\end{lemma}
\begin{proof} %
  We will adjust the proof of \cref{lem:lin_flat} by choosing the right spaces of holomorphic maps.
  We denote the open $\delta$-ball around $0$ in $\C^n$ by $\tilde \B_{\delta}$.
  Consider the vector space $\tilde W_\delta$ of $C^0$-families $(h_\sa)_{\sa\in \sA}$ of maps $(g_{\sa})_{\sa \in \sA}$ from $\tilde \B_{\delta}$ into $\C^n$ such that $\|D^r g\|$ is uniformly bounded. By $(H_1-H_2)$,
  we have that $f_{\sa}(\tilde \B_{\delta}) \subset \tilde \B_{\delta}$ for $\delta$ small enough. Hence the operator
  \begin{align*}
    T_f \colon \tilde W_{\delta} &\rightarrow \tilde W_\delta \\
    (h_{\sa})_{\sa \in \sA} & \mapsto (D_0f_{\sa})^{-1} \circ h_{\sigma(\sa)} \circ f_{\sa}
  \end{align*}
  is well-defined.
  Let $\tilde V_{\delta}$ be the subspace of $C^0$ families of holomorphic maps on $\tilde \B_\delta$ that satisfy $h_{\sa}(x) = o(x^r)$ for all $\sa \in \sA$. Then $T_f$ leaves $\tilde V_{\delta}$ invariant as we  assumed that $f$ is $r$-flat.

  We consider on $\tilde V_{\delta}$ the norm given by 
  \begin{equation}
    \|g\|_{V} = \norm{D^r g}_{0} \; .
    \label{eq:norm_adelta_holom}
  \end{equation}
  By Taylor--Lagrange expansion around 0, we see that for $k < r$:
\begin{equation}
\label{eqn:delta_bound_holomorphic}  \norm{ D^k  g(x) } \leq  \delta\cdot  \|g\|_{V}\; , \quad\forall x\in \tilde \B_\delta \; .
\end{equation}
This implies that  $(\tilde V_\delta, \|\cdot \|_{V})$ is a Banach space.  

Let us denote by $\tilde \iota \in \tilde W_\delta$ the constant family of the canonical inclusion $\tilde \B_\delta \subset \C^n$, so $\tilde \iota_{\sa}(x) = x$ for all $\sa \in \sA, x \in \tilde \B_{\delta}$.

A similar computation as in \cref{lem:lin_flat} shows that $T_f$ is a contraction on $\tilde V_{\delta}$ for $\delta$ small enough. So we find a fixed point
in $\iota + \tilde V_{\delta}$ that satisfies \eqref{eqn:lin_flat_holomorphic}. By restricting $\delta$ further, we can obtain that the resulting family is a family of biholomorphisms.

For holomorphic dependence, we get by \cref{lem:comp_holom} that $T_f$ depends holomorphically on $f$, and by the holomorphic case of \cref{lem:smooth_contraction} that the fixed point of $T_f$ on $\iota + \tilde V_{\delta}$ depends holomorphically on $f$.
\end{proof}
\begin{proof}[Proof of \cref{cor:lin_holomorphic}]
  This is now the direct composition of the two previous lemmas, together with a similar extension procedure as in the proof of \cref{thm:lin_no_param_smooth}.
\end{proof}
\begin{remark}
  The constructions in \cref{lem:formal_holomorphic} and \cref{lem:lin_flat_holomorphic} preserves real analyticity, so if the family $f_\sa$ is 
  real analytic (i.e.  $f_{\sa}(\tilde \B \cap \R^n) \subset \R^n$), then so are the families $(h_{\sa})$ are real. This shows \cref{rem:lin_real_analytic}.
\end{remark}
\subsection{Smooth dependence for main theorem}
\label{sec:dependence}
\label{sec:dependence_main}
The proof with parameter dependence follows the outline given in \cref{sec:non_param}, we provide a parametric version of \cref{lem:formal_no_param_smooth}
and \cref{lem:lin_flat_smooth}. 

Let $\cF \subset C^{0}(\sA,C^{\infty}(\B,\R^{n}))$ be the subspace of maps $(h_{\sa})_{\sa \in \sA}$ such that $h_{\sa}$ fixes 0 and has diagonal differential at $0$ for all $\sa \in \sA$. This is a closed vector subspace of $C^{0}(\sA,C^{\infty}(\B,\R^{n}))$ and so it is also a graded Fréchet space. Note that the set $\cU$ of maps satisfying the hypotheses $(H_1-H_2-H_3)$ of \cref{thm:lin_no_param_smooth} is an open subset in $\cF$.

We can now state the parametric version of \cref{lem:formal_no_param_smooth}:
\begin{lemma}
  The $C^0$-family of polynomial maps $(h_\sa)_{\sa \in \sA}$ of $\R^n$ with
  degree $\leq r$ constructed in \cref{lem:formal_no_param_smooth} depends
  smoothly tame on $f \in \cU$.
  \label{lem:formal_param}
\end{lemma}

\begin{proof}
  Let us first show that $(h_\sa)_{\sa \in \sA}$ depends smoothly on $f \in \cU$:
  as composition of $r$-jets is a smooth operation, the only thing left to show is that the fixed points of $O_{\mk, i}$ depend smoothly on $f \in \cU$. 
  As the operator $O_{\mk, i}$ depends smoothly on $f$, so does their fixed points by \cref{lem:smooth_contraction}. Finally, we use \cref{exa:smooth_banach} to promote smooth dependence to smooth tame dependence.
\end{proof}

Let $\cF_{r} \subset \cF$ be the subspace of maps 
$(h_{\sa})_{\sa \in \sA}$ 
such that $h_{\sa}(x) = D_{0} h_{\sa}(x) + o(x^r)$ for every $\sa \in \sA$. Note that these are all linear conditions, so $\cF_r$ is
really a closed vector subspace of $\cF$. We denote by $\cU_r$ the set of maps in $\cF_r$ satisfying the hypotheses $(H_1-H_2-H_3-H^r_4)$ of \cref{lem:lin_flat_smooth}. This is again an open subset of $\cF_r$.

Let us now state the parametric version of \cref{lem:lin_flat_smooth}:
\begin{lemma}
  If $\mathring f \in \cal U_r$. Then there exists a $\delta > 0$, and a $C^r$-neighborhood $\tilde \cU \subset \cal U_r$ of $\mathring f$ such that
  for every $f \in \cal U'$, there exists a unique
$C^0$-family $(h_{f,\sa})_{\sa\in \sA}$  of $C^\infty$-diffeomorphisms $h_{f,\sa}$ from the $\delta$-ball around 0  into $\R^n$ satisfying for following conditions for every $\sa\in \sA$:
\begin{equation}
  h_{f,\sa}(0) = 0\;, \quad h_{f,\sa}(x) = x + o(x^r)\;, \quad \text{and} \quad h_{f,\sigma(\sa)} \circ f_\sa(x)= D_0f_\sa \circ  h_{f,\sa}(x)\; .
  \label{eqn:lin_flat_smooth_param}
\end{equation}
Moreover, the map $g \mapsto h_{g, \sa}$ is tamely smooth.
  \label{lem:lin_flat_smooth_param}
\end{lemma}
We will proof this lemma in the rest of this section. Let us give a brief overview: the main idea is to follow the lines of the proof of \cref{lem:lin_flat}: we construct an operator $T_f$ and show it has a unique fixed point on some affine space, providing a solution to \cref{eqn:lin_flat_smooth_param}. To this end, we show that $\id - T_f$ is invertible with inverse $R_f$. Estimates similar to the proof of \cref{lem:comp_smooth_tame} intertwined with the contraction of $T_f$ on the $\| \cdot \|_r$-norm  are used to show that $R_f$ is tame. From this we conclude that the fixed point of $T_f$ depends smooth tame on $f$. We start with the construction of $T_f$:
\begin{lemma}
  Let $r \geq 2$, let $\mathring f \in \cU_r$. Then there exist $\delta' > 0$, $0 < C < C' < 1$, a $C^r$-neighborhood $\tilde \cU \subset \cU_r$ of $\mathring f$, and a family $(M_{j})_{j\geq r}$ of positive constants such that with
  
  \begin{equation*}
    \tilde V_{\delta'} = \left\{ h \in C^{0}(\sA,C^\infty(\B_{\delta'},\R^n)): h(x) = o(x^r) \right\}\; , 
  \end{equation*}
the following operator: 
  \begin{align*}
    T_{f} \colon    (h_{\sa})_{\sa \in \sA}\in \tilde V_{\delta'}   \mapsto  ((D_0 f_{\sa})^{-1} \circ h_{\sigma(\sa)} \circ f_{\sa})_{\sa\in  \sA} \in \tilde V_{\delta'}
  \end{align*}
  is well-defined for every $f \in \tilde \cU$ and $T \colon (f,h) \in \tilde \cU \times \tilde V_{\delta'}   \mapsto T_f(h) \in \tilde V_{\delta'}$ is smooth tame. Moreover $T_f$   satisfies the following bounds:
  \begin{equation}
    \|D^r T_{f}(h)\|_{0} \le C'\cdot \|D^r  h \|_{0} \; ,
    \label{eqn:contraction_r}
  \end{equation}
\begin{equation}
  \|D^j T_{f}(h)\|_{0} \le C\cdot \|D^j  h \|_{0} + M_j  \sum_{1 \leq k<j} \| f\|^{\frac{j-k}{j-1}}_{j} \cdot    \|D^k  h \|_{0}\;  \text{ for all } j \geq r\; .
  \label{eqn:contraction_higher_order}
\end{equation}
\label{lem:contraction_higher_order}
\end{lemma}
\begin{proof}
  The construction of $\delta',$ $C$ and $M$ in \cref{lem:lin_flat} depended only on $C^r$-bounds of $f$. So
  we can find a $C^r$-neighborhood $\tilde \cU \subset \cU_r$ of $f_0$, $\delta' > 0$, $0 < C < C' < 1$ such
  that the following holds:
  \begin{itemize}
    \item for every $f \in \tilde \cU$, $\sa \in \sA$, the map $f_{\sa}$ is contracting on $\B_{\delta'}$,
    \item with $\tilde \mu_{\sa} = \inf_{f \in \tilde \cU} \min_{i} \lambda_{f, \sa, i}$, we have
\[C\cdot \tilde \mu_\sa > \|D_x f _\sa\| ^r\; , \quad \forall  x\in \B_{\delta'}\; .\]
\item The operator $T_f \colon V_{\delta'} \rightarrow V_{\delta'}$ is $C'$-contracting for all $f \in \tilde \cU$, where
  $V_{\delta'}$ is the Banach space 
    $V_{\delta'} = \left\{ h \in C^{0}(\sA,C^r(\B_{\delta'},\R^n)): h(x) = o(x^r) \right\},$ 
  \end{itemize}

In particular $\|D_x f _\sa\|<1 $  for every $x\in \B_{\delta'}$. Thus for every $j\ge r$, we have:
\begin{equation}
\label{eqn:def_C_param}
  C > \tilde \mu_\sa^{-1} \cdot  \|D_x f _\sa\| ^j\; , \quad \forall  x\in \B_{\delta'}, f \in \tilde \cU \; .
\end{equation} 
Since every $f_{\sa}$ is contracting on $\B_{\delta'}$ and fixes 0, we see that $T$ is well-defined. Similar to \cref{lem:comp_smooth_tame} one shows 
  that $T$ is smooth tame.
  Now \cref{eqn:contraction_r} follows from the $C'$-contraction on $V_{\delta'}$.

  We will again use explicit Faà di Bruno's formula \cref{eqn:faa_di_bruno}: there are constants
$c_{j,k,i_1,\dots,i_k} \in \Z$ such that
  \begin{equation*}
  D^j  (h_{\sigma(\sa)}\circ f_{\sa})  =  D^j_{f_{\sa}}  h_{\sigma(\sa)}  ( Df_{\sa})^{\otimes j}+  
  \sum_{k<  j} \sum_{i_1 +  \dots + i_k = j} c_{j,k,i_1,\dots,i_k} D^k_{f_{\sa}}  h_{\sigma(\sa)} \left( D^{i_1}f_{\sa} \otimes \dots  \otimes D^{i_k}f_{\sa} \right)\; .
\end{equation*}

By \cref{lem:interpolation} and that $\|D^{1} f\|_{0} < 1$, there is a positive constant $I_j$ such that for all $f \in \cU$ we have $\|D^{i}f_{\sa}\|_{0} \leq I_j \|f_{\sa}\|_{j}^{\frac{i-1}{j-1}}$. So for $i_1 + \dots + i_k = j$ we can bound 
\begin{equation*}
  \|D^{i_1}f_{\sa}\|_{0} \cdots \|D^{i_k}f_{\sa}\|_{0} \leq I^k_j \|f_{\sa}\|_{j}^{\frac{j-k}{j-1}} \; .
\end{equation*}
So we can find constants $M_j > 0$ such that 
\begin{equation*}
  \sum_{i_1 +  \dots + i_k = j} \| c_{j,k,i_1,\dots,i_k} D^k_{f_{\sa}}  h_{\sigma(\sa)} \left( D^{i_1}f_{\sa} \otimes \dots  \otimes D^{i_k}f_{\sa} \right) \|_0 \leq \tilde \mu_{\sa} M_j \| f\|^{\frac{j-k}{j-1}}_{j} \| D^k h_{\sigma(\sa)}\|_0 \; .
\end{equation*}
     This gives:
  \begin{equation*}
\|  D^j  (h_{\sigma(\sa)}\circ f_{\sa}) \|_{0}   \leq  \| D^j_{f_{\sa}}  h_{\sigma(\sa)}  ( Df_{\sa})^{\otimes j}\|+  
\tilde \mu_{\sa} \cdot   M_j  \sum_{k<  j} \| f\|^{\frac{j-k}{j-1}}_{j}
\| D^k h_{\sigma(\sa)}\|_0 \; .
\end{equation*}
 Consequently:
   \begin{equation*}
     \|D^j T_f(h)_\sa \|_{0} \le  \tilde \mu_{\sa}^{-1}  \| Df_{\sa} |\B_{\delta'}  \|_{0}^j \cdot \|  D^j  h_{\sigma(\sa)} \|_{0}  +      M_j  \sum_{k<  j} \|f\|^{\frac{j-k}{j-1}}_{j} 
     \| D^k h_{\sigma(\sa)}\|_0 \; .
   \end{equation*}
From \cref{eqn:def_C_param}  we obtain the sought result. 
\end{proof} 
\begin{lemma}
  Let $\tilde U, \delta', C, C', M_j$ as resulting from \cref{lem:contraction_higher_order}. The following map is well-defined and tame:
\[R: (f,h)\in \tilde \cU \times \tilde V_{\delta'} \mapsto R_f(h):=:\sum_{m\ge 0}   T^m_f(h) \in \tilde V_{\delta'} \; .\]
In particular, $\id - T_f$ is invertible on $\tilde V_{\delta'}$ with inverse $R_f(h)$.
  \label{lem:r_well_defined}
\end{lemma}
\begin{proof}
   We already know that every summand $T^k_f(h)$ is well-defined.
  We will show the following:
  there is a family $(\tilde M_{j})_{j \geq r}$ of positive constants (depending on $\tilde \cU, \delta', C, M_j$, but not on $f \in \tilde \cU$) such that 
  \begin{equation}
    \sum_{m \ge 0} \|T^m_f(h)\|_{j} \leq \tilde M_j(\| h \|_{r} \|f \|_{j}  +  \| h \|_{j}) \text{ for all }f \in \tilde U, h \in \tilde V_{\delta'}, j \geq r \; .
    \label{eqn:def_tilde_m}
  \end{equation}
  With this it is clear that $R_f(h)$ is well-defined, and equal to the inverse of $\id - T_f$ on $V_{\delta'}$.
    As we can bound $\|h\|_{r}$ in a $C^r$-neighborhood of $(f,h)$, this then also shows that $R_f(h)$ is tame.

    Let us inductively show the existence of $(\tilde M_j)_{j \geq r}$ satisfying \eqref{eqn:def_tilde_m}.
  Let $s_j = \sum_{m\geq 0} \| D^j T^m_{f}(h)\|_{0}$. 
    For $j \leq r$ we can bound 
    \begin{equation*}
      \| D^j T^m_{f}(h)\|_{0} \leq \| D^r T^m_{f}(h)\|_{0} \leq C'^m \| h \|_{r}
    \end{equation*}
    by \eqref{eqn:delta_bound} and \eqref{eqn:contraction_r} and so 
    \begin{equation}
      s_j \leq s_r =  \sum_{m \ge 0} \|T^m_f(h)\|_{r} \leq \frac{1}{1-C'} \| h \|_{r} \quad \forall j \leq r \;.
      \label{eqn:lower_s_bound}
    \end{equation}
    In particular we can start with $\tilde M_r = \frac{1}{1-C'}$.

    For $j > r$, we note that 
    \[
      \sum_{m\geq 0} \| T^m_{f}(h) \|_{j} \leq s_j + \sum_{m\geq 0} \| T^m_{f}(h) \|_{j-1} \leq s_j + \tilde M_{j-1}(\| h \|_{r} \|f \|_{j-1}  +  \| h \|_{j-1}) \leq s_j + \tilde M_{j-1}(\| h \|_{r} \|f \|_{j}  +  \| h \|_{j}) 
    \]
    since the norms $\|\cdot\|_{j}$ are an increasing family. So it is enough to bound $s_j$ by a constant times $(\| h \|_{r} \|f \|_{j}  +  \| h \|_{j})$.

    We will first show the following bound:
    \begin{equation}
      s_j \leq \frac{\| h \|_{j} +   M_j  \sum_{1 \leq k < j}\| f \|^{\frac{j-k}{j-1}}_{j} s_k}{1-C} \; .
      \label{eqn:s_rec_bound}
    \end{equation}

    For this, we first note that we can obtain
    \begin{equation}
      \| D^j T^m_f(h) \|_0 \leq C^m \|D^j h\|_{0} + \sum_{0\leq l<m} C^{m-l-1} M_j \sum_{1\leq k < j} \| f \|^{\frac{j-k}{j-1}}_{j}  \| D^k T^l_f(h) \|_0 \quad \forall m \geq 0
      \label{eqn:explicit_generating_series_bound}
    \end{equation}
    by induction over $m$ from \eqref{eqn:contraction_higher_order}.

    Summing up \eqref{eqn:explicit_generating_series_bound} we obtain
    \begin{align}
      s_j  &= \sum_{m\geq 0} \| D^j T^m_f(h) \|_0 \leq \sum_{m \geq 0} 
      \left(C^m \|D^j (h)\|_{0} + \sum_{0\leq l<m} C^{m-l-1} M_j \sum_{1\leq k < j} \| f \|^{\frac{j-k}{j-1}}_{j}  \| D^k T^l_f(h) \|_0 \right) \\
 &= \frac{1}{1-C} \|D^j (h)\|_{0} +  \sum_{m\geq 0} \sum_{0\leq l<m} C^{m-l-1} M_j \sum_{1\leq k < j} \| f \|^{\frac{j-k}{j-1}}_{j}  \| D^k T^l_f(h) \|_0  \\
 \label{eqn:before_sub}
  &= \frac{1}{1-C} \|D^j (h)\|_{0} +  \sum_{m' \geq 0} C^{m'} M_j \sum_{1\leq k < j} \| f \|^{\frac{j-k}{j-1}}_{j} \sum_{l\geq 0}  \| D^k T^l_f(h) \|_0  \\
 \label{eqn:after_sub}
 &= \frac{1}{1-C} (\|D^j (h)\|_{0} +  M_j  \sum_{1\leq k < j}\| f \|^{\frac{j-k}{j-1}}_{j} s_j) \\
 &\leq \frac{1}{1-C} (\|h\|_{j} + M_j  \sum_{1\leq k < j} \| f \|^{\frac{j-k}{j-1}}_{j} s_j) \; 
    \end{align}
    where we replaced the summation index $m$ by $m' \coloneqq m-l-1$ from \eqref{eqn:before_sub} to \eqref{eqn:after_sub}. Since all terms in our summations are nonnegative, we can freely change the summation order.
    This establishes \eqref{eqn:s_rec_bound}.

    Let us now insert \eqref{eqn:lower_s_bound} for $1\leq k \leq r$ and inductively 
    \eqref{eqn:def_tilde_m} combined with $s_k \leq \sum_{m \ge 0} \|T^m_f(h)\|_{j}$ for $r < k < j$ into \eqref{eqn:s_rec_bound} to obtain
    \begin{align}
      s_j &\leq \frac{\| h \|_{j} +   \sum_{1 \leq k < j} M_j \| f \|^{\frac{j-k}{j-1}}_{j} s_k}{1-C}  \\
      &\leq \frac{1}{1-C} (\| h \|_{j} +  \sum_{1 \leq k \leq r} M_j \| f \|^{\frac{j-k}{j-1}}_{j} \tilde M_r \| h \|_{r} +  \sum_{1 \leq k < j} M_j \| f \|^{\frac{j-k}{j-1}}_{j} \tilde M_k( \| h\|_{r}\| f \|_{k} + \| h \|_{k})) \; .
      \label{eqn:s_j_induction}
    \end{align}
We can bound every monomial appearing in the right hand side of \eqref{eqn:s_j_induction} by a constant times $\| h \|_{r} \|f \|_{j}  +  \| h \|_{j}$:
Trivially, $\|h\|_j \leq \| h \|_{r} \|f \|_{j}  +  \| h \|_{j}$. We also have
\begin{equation*}
  \| f \|^{\frac{j-k}{j-1}}_{j} \| h \|_{r} \leq (\| f \|_{j} + 1) \| h \|_{r} \leq \| f \|_{j} \| h \|_{r} + \| h \|_{j} \; .
\end{equation*}
  By \cref{lem:interpolation} there is a $I_j > 0$ such that
    \begin{equation*}
      \| f \|_{k} \leq  I_j \| f \|^\frac{k-1}{j-1}_{j} \| f \|^\frac{j-k}{j-1}_{1}\; , \quad
      \| h \|_{k} \leq  I_j \| h \|^\frac{k-1}{j-1}_{j} \| h \|^\frac{j-k}{j-1}_{1} \quad \forall f \in \tilde \cU, h \in \tilde V_{\delta'}.
    \end{equation*}
    So we have the following:
    \begin{equation*}
      \| f \|^{\frac{j-k}{j-1}}_{j} \| h\|_{r}\| f \|_{k} \leq I_j \| f \|^{\frac{j-k}{j-1}}_{j} \| h\|_{r} \| f \|^\frac{k-1}{j-1}_{j} \| f \|^\frac{j-k}{j-1}_{1} = I_j \| f \|_{j} \| h\|_{r} \| f \|^\frac{j-k}{j-1}_{1} \leq I_j \| f \|_{j} \| h\|_{r} \\
    \end{equation*}
    where we also used $\| f \|_1 < 1$. Finally
    \begin{equation*}
    \| f \|^{\frac{j-k}{j-1}}_{j} \| h \|_{k} \leq I_j (\| f \|^{\frac{j-k}{j-1}}_{j} \|h\|^{\frac{j-k}{j-1}}_{1}) \|h\|^{\frac{k-1}{j-1}}_{j} \leq
      I_j (\| f \|_{j} \| h \|_{1} + \| h \|_{j}) \leq I_j (\| f \|_{j} \| h \|_{r} + \| h \|_{j}) 
    \end{equation*}
    using also \eqref{eqn:coarse_amgm}.
  So we can bound every monomial appearing in the right hand side of \eqref{eqn:s_j_induction} by a constant times $\| h \|_{r} \|f \|_{j}  +  \| h \|_{j}$. From this, the existence of $\tilde M_j$ follows.
\end{proof}
\begin{lemma}
  Under the assumptions of 
  \cref{lem:r_well_defined}, $R$ is smooth tame.
\end{lemma}
\begin{proof}
  Let us first show that $R$ is $C^1$. We follow the proof strategy of \cite[Theorem I.5.3.1]{Hamilton82}.
  As $R$ is linear in $h$, we only have to show that $R$ is $C^1$ in $f$. Let us compute
  \begin{align*}
    \frac{R_{f+tg}h - R_f h}{t} &= \frac{R_{f+tg}\left(  (\id - T_f)R_{f} h - (\id - T_{f+tg})R_f h  \right)}{t}
    &= \frac{R_{f+tg}\left( T_{f+tg}R_f h  - T_f R_{f} h \right)}{t} \; .
  \end{align*}
  For $t \rightarrow 0$ we have 
  $\frac{\left( T_{f+tg}R_f h  - T_f R_{f} h \right)}{t} \rightarrow \partial^1 T_f(g) R_f h$ as $T$ is $C^1$.
  Since we know that $R$ is continuous, we obtain
  \begin{align*}
    \frac{R_{f+tg}\left( T_{f+tg}R_f h  - T_f R_{f} h \right)}{t} \rightarrow R_f \partial^1 T_f(g) R_f h \; .
  \end{align*}

  As $T$ is smooth tame, and $R$ is tame, it follows from this formula and the chain rule that all higher derivatives
  of $R$ exist and are tame, so $R$ is also smooth tame.
\end{proof}
\begin{proof}[Proof of \cref{lem:lin_flat_smooth_param}]
  We use $\tilde \cU$, $\delta'$, $T$ and $R$ as in the previous lemmas. As in the proof of \cref{lem:lin_flat}, we let $\iota \in C^{0}(\sA, C^{\infty}(\B_{\delta'},\R^n))$ 
 be the constant family of the canonical inclusion $\B \subset \R^n$, so $\iota_{\sa}(x) = x$ for all $\sa \in \sA, x \in \B_{\delta'}$.
 The operator $T_f$ also acts on the affine space $\iota + \tilde V_{\delta'}$, as $T_f(\iota) \in \iota + \tilde V_{\delta'}$. Since the operator $\id - T_f$ has inverse $R_f$ on $\tilde V_{\delta'}$, we know that there is a unique fixed point in $\iota + \tilde V_{\delta'}$. This fixed point satisfies \eqref{eqn:lin_flat_smooth_param}.

 In fact, we have the follwing explicit formula for the fixed point:
 \[ h_{f} = \iota - R_{f}(\iota - T_f(\iota)) \; .\]
 To see this, we show that $(\id - T_f)h_f$ vanishes:
 \begin{equation*}
   (\id - T_f)h_f = (\id - T_f)(\iota) - (\id - T_f) R_{f} (\iota - T_f(\iota)) = \iota - T_f - (\iota - T_f(\iota)) = 0 
 \end{equation*}
 As $R$ and $T$ are tamely smooth by \cref{lem:contraction_higher_order} and \cref{lem:lin_flat_smooth_param}, $h$ depends tamely smooth on $f$ as a family of $C^{\infty}$-maps on the $\delta'$-ball around 0. By the local inverse function theorem and restricting $\tilde \cU$ further if necessary, we thereby constructed a family of $C^{\infty}$-diffeomorphisms. Uniqueness already follows from the uniqueness of \cref{lem:lin_flat}.
\end{proof}
\begin{proof}[Proof of \cref{thm:lin_no_param_smooth}]
  We follow the proof of Corollary~\ref{cor:lin_no_param} using our smooth version of the lemmas. We obtain smooth tame dependence
  of $h$ on a small ball $\B_{\delta'}$. Note that the extension process used in the proof of Corollary~\ref{cor:lin_no_param} is also smooth tame.

\end{proof}

\appendix
\section{Smooth dependence of contractions}
Here we state smooth dependence of fixed points of contractions. We follow the lines of \cite[Appendix A]{Yoccoz95} that we extend to the case of a complex Banach and a Fréchet parameter space.
\begin{lemma}
  Let $U$ be an open subspace of a Fréchet space $\cF$, let $V$ be a open subset of an Banach space $\cB$. Let $\rho \colon U \times \overline V \rightarrow V$ be a continuous map such that for every $u \in U$, the map $\rho_u \colon \overline V \rightarrow V$ is a $k$-contraction for a fixed $k \in (0,1)$ and fixed point $\phi(u) \in V$. %

  Then the map $\phi \colon U \rightarrow V$ is continuous. Moreover, if $\rho$ is $C^r$ or smooth, then so $\phi$. If $\cF$ and $\cB$ are complex Banach spaces and $\rho$ is holomorphic, then $\phi$ is also holomorphic.
  \label{lem:smooth_contraction}
\end{lemma}
\begin{proof}

  We first observe that $\phi$ is continuous. By the proof of the Banach fixed point theorem we have
  \[ \|v - \phi(u)\| \leq \sum^{\infty}_{i=0} \| \rho^i_{u}(v) - \rho^{i+1}_{u}(v) \| \leq \frac{1}{1-k} \| v - \rho(u,v) \| \quad \forall u \in U, v \in V \; . \]
  In particular we have
  \[ \| \phi(u') - \phi(u) \| \leq \frac{1}{1-k} \| \phi(u') - \rho(u, \phi(u')) \| = \frac{1}{1-k} \| \rho(u', \phi(u')) - \rho(u, \phi(u')) \| \; . \]
  So the continuity of $\phi$ at $u'$ follows from the continuity of $\rho$ at $(u', \phi(u'))$.

  Let us now further suppose that $\rho$ is $C^1$.
  We will compute the differential of $\phi$ as follows: 
  let $u \in U$, $h \in \cF$. By continuity of $\phi$ we know that for $|t|$ small enough, we have that the segment
  $[u, u+t\cdot h]$ is in $U$ and $[\phi(u),\phi(u+t\cdot h)]$ is in $V$.
  Let $\gamma \colon [0,1] \rightarrow U \times V$ be given by $\gamma(s) = (u + s\cdot t \cdot h, s \cdot \phi(u+t\cdot h) + (1 - s) \cdot \phi(u))$.
  We can then compute

  \begin{align*}
   &\phi(u + h\cdot t) - \phi(u) = \rho(u+h\cdot t,\phi(u+h\cdot t)) - \rho(u,\phi(u))  \\
   = &\rho(\gamma(1)) - \rho(\gamma(0)) =
   \int^1_{0} D_{\gamma(s)} \rho (\gamma'(s)) d s =
    \int^1_{0} D_{\gamma(s)}^{(1)} \rho (h \cdot t) + D_{\gamma(s)}^{(2)} \rho (\phi(u+t \cdot h) - \phi(u)) d s  \; .
  \end{align*}
  Here $D_{\gamma(s)}^{(1)} \rho \colon \cF \rightarrow \cB,  D_{\gamma(s)}^{(2)} \rho \colon \cB \rightarrow \cB$, are the partial derivatives of
  $\rho$ at $\gamma(s)$. We note that $D_{\gamma(s)}^{(2)} \rho$ has norm bounded by $k$. In particular, $1 -  \int^1_{0} D^{(2)}_{\gamma(s)} \rho ds$
  is an invertible operator on $\cB$. 
  Hence we obtain
  \begin{equation*}
    \frac{\phi(u + h \cdot t) - \phi(u)}{t} = (1 - \int^1_{0} D^{(2)}_{\gamma(s)} \rho ds )^{-1} \int^1_{0} D^{(1)}_{\gamma(s)} \rho (h) ds  \; .
\end{equation*}

  As $\rho$ is $C^1$, the right hand side converges to
  $(1 - D_{(u, \phi(u))}^{(2)} \rho)^{-1} D_{(u, \phi(u))}^{(1)} \rho(h)$ for $t \rightarrow 0$. This is a continuous function in $u$,
    so $\phi$ is also $C^1$.

    From the formula the statements for higher regularity or holomorphic maps follow directly.
\end{proof}
\bibliographystyle{alpha}
\bibliography{bibfile.bib}
\Addresses

\end{document}